\documentclass[11pt, reqno]{amsart}   	
\usepackage[truedimen,margin=29truemm]{geometry}          		
\usepackage{graphicx}				
\usepackage{amssymb}
\usepackage{amsmath}
\usepackage{amsthm}
\usepackage{url}
\usepackage{enumerate} 

\newcommand{\WF}{\mathop{\mathrm{WF}}}
\newcommand{\VOL}{\mathop{\mathrm{Vol}}}
\newcommand{\ccc}{\mathop{\mathrm{c}}}
\newcommand{\low}{\mathop{\mathrm{low}}}
\newcommand{\dimm}{\mathop{\mathrm{dim}}}
\newcommand{\loc}{\mathop{\mathrm{loc}}}
\newcommand{\diag}{\mathop{\mathrm{diag}}}
\newcommand{\deter}{\mathop{\mathrm{det}}}
\newcommand{\Iden}{\mathop{\mathrm{Id}}}

\newcommand{\sgn}{\mathop{\mathrm{sgn}}} 
\newcommand{\ac}{\mathrm{ac}} 
\newcommand{\pp}{\mathrm{pp}} 

\newcommand{\hm}{\mathrm{HM}}
\newcommand{\Op}{\mathrm{Op}}
\newcommand{\re}{\mathop{\mathrm{Re}}} 
\newcommand{\im}{\mathop{\mathrm{Im}}} 
\newcommand{\tr}{\mathop{\mathrm{Tr}}}
\newcommand{\supp}{\mathop{\mathrm{supp}}}

\newcommand{\N}{\mathbb{N}} 

\newcommand{\R}{\mathbb{R}} 
\newcommand{\C}{\mathbb{C}} 
\newcommand{\T}{\mathbb{T}}

\numberwithin{equation}{section}

\theoremstyle{plain}
\newtheorem{thm}{Theorem}[section]
\newtheorem{proposition}[thm]{Proposition}
\newtheorem{lemma}[thm]{Lemma} 
\newtheorem{corollary}[thm]{Corollary}

\theoremstyle{definition} 

\newtheorem{defn}[thm]{Definition}

\newtheorem{assump}[thm]{Assumption}

 \newtheorem{remark}[thm]{Remark}

 \newtheorem*{remarks*}{Remarks}
\newtheorem*{remark*}{Remark}

\title[Semiclassical limit on scattering manifolds]{Semiclassical limit of orthonormal Strichartz estimates on scattering manifolds}
\author{Akitoshi Hoshiya}

\begin{document}
\maketitle

\begin{abstract}
We study a quantum and classical correspondence related to the Strichartz estimates. First we consider the orthonormal Strichartz estimates on manifolds with ends. Under the nontrapping condition we prove the global-in-time estimates on manifolds with asymptotically conic ends or with asymptotically hyperbolic ends. Then we show that, for a class of pseudodifferential operators including the Laplace-Beltrami operator on the scattering manifolds, such estimates imply the global-in-time Strichartz estimates for the kinetic transport equations in the semiclassical limit. As a byproduct we prove that the existence of a periodic stable geodesic breaks the orthonormal Strichartz estimates. In the proof we do not need any quasimode. As an application we show the small data scattering for the cutoff Boltzmann equation on nontrapping scattering manifolds.  
\end{abstract}

\section{Introduction and main results}\label{23110221}
\subsection{Introduction}\label{25371356}
The purpose of this paper is twofold. The first one is to see some effects of geometry on long-time behavior of solutions to the Schr\"odinger equation, especially the effects by trapped sets and conjugate points. The other is to see quantum and classical correspondence for the Strichartz estimates. As explained later in this section, these topics are not independent, but closely related with each other.

Let $(X, g)$ be a noncompact Riemannian manifold and $P$ be a self-adjoint operator on $L^2 (X, dg)$. The Strichartz estimates for the Schr\"odinger equation:
\begin{align}
\left\{
\begin{array}{l}
i\partial_t u (t)= Pu(t), \\
u(0)= u_{0}
\end{array}
\right.
\label{25381311}
\end{align}
are linear estimates for the solution $e^{-itP} u_0$ to (\ref{25381311}) in the Lebesgue spaces. Typical examples are inequalities of the following type (sometimes with a derivative loss).
\begin{align}
\|e^{-itP} u_0\|_{L^q _t L^r _z} \lesssim \|u_0\|_2.\label{25381329}
\end{align}
Here $\|\cdot\|_p = \|\cdot\|_{L^p (X, dg)}$ and $L^q _t L^r _z = L^q (\R_t; L^r (X, dg(z)))$ denote the standard Lebesgue norm and the Bochner space for $p, q, r \in [1, \infty]$. The important points are that (\ref{25381329}) represents the spacetime decay of the solution to (\ref{25381311}) if $q, r < \infty$ and also that it implies the smoothing effect by the propagator if $r >2$. The latter is a consequence of the fact that the initial data $u_0$ is only in $L^2 (X, dg)$ but the solution satisfies $u(t) \in L^r (X, dg)$ for almost all $t \in \R$. One of the simplest examples of $(X, g)$ and $P$ satisfying (\ref{25381329}) is the Euclidean space with the standard metric: $(\R^d, dz^2)$ and the Laplacian. It is shown in \cite{St, GV, Y, KT} that (\ref{25381329}) holds in the present situation if and only if $(q, r)$ satisfies the admissible condition: $q, r \in [2, \infty]$, $\frac{2}{q} = d(\frac{1}{2} - \frac{1}{r})$, $(d, q, r) \neq (2, 2, \infty)$. 
These results are generalized to some kinds of (not necessarily Riemannian) manifolds but it is known that geometry has large effects on (\ref{25381329}). In \cite{BGT} the (local-in-time) Strichartz estimates for the Laplace-Beltrami operator on general compact Riemannian manifolds are considered and in general there are derivative loss of order $\frac{1}{q}$. Similar results for nonelliptic operators are proved in \cite{MT}. Such derivative loss is a consequence of the fact that all geodesics on compact manifolds are trapped. Similar situations also happen on noncompact manifolds. On asymptotically Euclidean spaces, both the local and global-in-time Strichartz estimates are considered in \cite{BT1, BT2, Di, MMT, RZ, ST, Tat} without derivative loss under the nontrapping condition. Some of the above results also prove the Strichartz estimates without nontrapping conditions but instead have a derivative loss due to a trapped set. See \cite{Tai1} for similar results for operators associated with nondegenerate metrics. On scattering manifolds, or manifolds with asymptotically conic ends, the local-in-time estimates are proved in \cite{HTW, M1} for nontrapping metrics and in \cite{BGH} in the presence of a mild trapped set. The global-in-time estimates are considered in \cite{HZ} or in \cite{BM, ZZ2} for nontrapping or mild-trapping metrics respectively. Contrary to the above results, it is also commented without proof in \cite{BGH} that the existence of an elliptic stable non-degenerate periodic geodesic breaks the Strichartz estimates without loss (hence global-in-time estimates never hold). Such a phenomenon is a consequence of the existence of quasimodes concentrating on the periodic geodesic. See \cite{Ra, Chr1}, \cite{Chr2} or \cite{RT} for quasimodes construction by the WKB method, by the Weyl law or by the spherical harmonic concentrating on the equator of the sphere. We remark that \cite{Chr1} uses such quasimodes to show the sharpness of the Strichartz estimates on some warped product manifolds.

Note that the effect of a trapped set is not limited to the Strichartz estimates. In \cite{Do1} it is proved that on a connected complete Riemannian manifold, the local smoothing estimates:
\begin{align}
\|e^{it\Delta_{g}} u_0\|_{L^2 _{\loc} H^{\frac{1}{2}} _{\loc}} \lesssim \|u_0\|_2 \label{25313947}
\end{align}
hold if and only if the manifold is nontrapping. (\ref{25313947}) is often used in the proof of the (local-in-time) Strichartz estimates and the derivative loss in (\ref{25313947}) directly appears as a loss in the Strichartz estimates. However if the trapped set is mild in the sense of \cite{NZ, BGH}, the derivative loss in (\ref{25313947}) is at most logarithmic. Combining it with a semiclassical Strichartz estimate on a time interval of length $h |\log h|$ (\cite[Theorem 3.8]{BGH}), the Strichartz estimates in \cite{BGH} have no loss. Similar situations also occur concerning the smoothness of the fundamental solution $e^{it\Delta_{g}} (z, z')$. As in \cite{Do2, Tai2}, $e^{it\Delta_{g}} (z, z'), t \neq 0$ is smooth if the manifold is nontrapping or mild-trapping but not smooth for any $t \neq 0$ if there exists a certain quasimode (see \cite[Proposition 3.2]{Tai2} for the precise statement).

The Strichartz estimates are also considered on negatively curved manifolds. On nontrapping asymptotically hyperbolic manifolds the local or global-in-time Strichartz estimates are proved in \cite{B2} or in \cite{Che1} respectively (see also \cite{AP} for the exact hyperbolic space). In \cite{BGH} the Strichartz estimates without loss are proved on convex co-compact hyperbolic manifolds under the condition that the Hausdorff dimension of their limit sets are small. Without such smallness assumption it is also shown in \cite{W} that for the surface the derivative loss for the local-in-time estimates are arbitrarily small. Though the above positive results are considered on manifolds with funnel ends, it is shown in \cite{B1} that no local-in-time Strichartz estimates hold on surfaces with cusp ends. For generalizations to the Damek-Ricci spaces, which include symmetric spaces of noncompact type and rank one, we refer to \cite[Introduction]{B2}. 

In addition to trapped sets it is also known that conjugate points are sensitive to dispersive properties of solutions to the Schr\"odinger equation. In \cite{HW}, it is shown that on nontrapping scattering manifolds, if $(z, z')$ is a conjugate point (this means there exists a geodesic from $z$ to $z'$ such that $z'$ is conjugate to $z$), the following dispersive estimate may fail:
\begin{align*}
|e^{it\Delta_{g}} (z, z')| \lesssim |t|^{-\frac{d}{2}}.
\end{align*}
This is in contrast with the standard Euclidean space, which has no conjugate point and satisfies the dispersive estimates. Related to the effect by conjugate points, see \cite{JZ} for the dispersive estimates for the Schr\"odinger equation on a product cone $(0, \infty) \times Y$ where the conjugate radius of the closed manifold $Y$ is larger than $\pi$. Recently, in the case $Y = \rho \mathbb{S}^{d-1}$ with $\rho >0$ and $d \ge 3$, the dispersive estimates are proved under $\rho \ge 1$ (conjugate radius $\ge \pi$) and disproved under $\rho < 1$ and $\frac{1}{\rho} \notin 2 \N$ in \cite{Tai3}.

As a final remark on the known results, we refer to \cite{Che2} for the Strichartz estimates on noncompact manifolds with multiple conical singularities (without conjugate points) and to \cite{AP2, SSWZ, Zh1, ZZ1, ZZ3} for the Strichartz estimates for the wave and Klein-Gordon equations in similar geometric settings.

Now the first purpose of this paper is to study the effects by trapped sets and conjugate points on \textit{the orthonormal Strichartz estimates}, which are extensions of the ordinary Strichartz estimates and originated in \cite{FLLS}. For the free Laplacian on $\R^d$, they are the following inequalities:
\begin{align}
\left \| \sum_{j=0}^ \infty{\nu_j |e^{it\Delta}f_j|^2} \right\|_{L^{\frac{q}{2}} _t L^{\frac{r}{2}} _z} \lesssim \| \nu\|_{\ell^{\beta}},
\label{25314906}
\end{align}
where $\{f_j\}$ is any orthonormal system in $L^2 (\R^d)$ and $\nu = \{\nu_j\}$ is any complex-valued sequence. The exponent $(q, r)$ satisfies the admissible condition and $\beta \in [1, \infty)$ varies depending on $(q, r)$. If $d \ge 2$ and $r \in (2, \frac{2d}{d-2})$ we can take $\beta >1$ (see \cite{BHLNS}). One of the important points is that (\ref{25314906}) with $\beta =1$ is actually equivalent to (\ref{25381329}) with $P=-\Delta$. By the embedding of $\ell^{\beta}$ space, (\ref{25314906}) with $\beta >1$ is stronger than the ordinary Strichartz estimates. There are many results on the orthonormal Strichartz estimates if the Hamiltonian is a Fourier multiplier on $\R^d$. Furthermore (\ref{25314906}) is extended to the Schr\"odinger operator on $\R^d$ with general decaying potentials. We refer to \cite[{\S}1]{H3} for such results. However there are few results on manifolds except for $\R^d$. In \cite{N} the orthonormal Strichartz estimates are considered on $\T^d$ and recently extended to general compact manifolds in \cite{WZZ}. Concerning noncompact spaces, convex co-compact hyperbolic manifolds are considered in \cite{H3} but there seems to be no other result. In this paper we consider the orthonormal Strichartz estimates on manifolds with asymptotically conic ends or  with asymptotically hyperbolic ends. Positive results are explained in Subsection \ref{25371400} and negative results are introduced in Subsection \ref{25371403}.

Next we move on to the second purpose, quantum and classical correspondence. We recall that (\ref{25314906}) is an estimate for the solution to the Heisenberg equation. For a bounded operator $\gamma \in \mathcal{B} (L^2 (\R^d))$, $\rho _{\gamma} \in L^{1} _{\loc} (\R^d)$, \textit{the density} of $\gamma$ is formally defined by $\rho _{\gamma} (z) = K_{\gamma} (z, z)$, where $K_{\gamma} \in \mathcal{S}' (\R^{2d})$ is the Schwartz kernel of $\gamma$ (more precisely if there is a $f \in  L^{1} _{\loc} (\R^d)$ such that $\tr (\chi \gamma) = \int_{\R^d} f(z) \chi (z)dz$ for any characteristic function $\chi$ of a positive measure subset, $f$ is denoted by $\rho _{\gamma}$ or $\rho (\gamma)$ and called the density). By a direct computation, if we define $\gamma_{0} := \sum_{j=0}^{\infty} \nu_{j} | f_j \rangle \langle f_j | \in \mathfrak{S}^{\beta} (L^2 (\R^d))$ for $\nu = \{\nu_{j}\} \in \ell^{\beta}$ with $\beta \in [1, \infty)$ and for an orthonormal system $\{f_j\} \subset L^2 (\R^d)$, we have $\rho (e^{it\Delta} \gamma_{0} e^{-it\Delta}) = \sum_{j=0}^ \infty{\nu_j |e^{it\Delta}f_j|^2}$. Therefore (\ref{25314906}) is rewritten as
\begin{align}
\|\rho (e^{it\Delta} \gamma_{0} e^{-it\Delta})\|_{L^{\frac{q}{2}} _t L^{\frac{r}{2}} _z} \lesssim \|\gamma_0\|_{\mathfrak{S}^{\beta}}. \label{253161706}
\end{align}
Here $\mathfrak{S}^{\beta}$ denotes the Schatten class (see \cite{Si}). Since $e^{it\Delta} \gamma_{0} e^{-it\Delta}$ is a solution to the Heisenberg equation, (\ref{25314906}) is a smoothing estimate for the Heisenberg equation (restricted to the diagonal) in some sense. This rewriting also holds for general self-adjoint operators $P$, not limited to $\Delta$, especially for any self-adjoint operator $P=-\partial_{i} g^{ij} (z) \partial_{j}$ where $g^{ij} \in S^{0} (\R^d) := S \left(1, \frac{dz^2}{\langle z \rangle^2} \right)$. Furthermore, by the asymptotic formula $[a^w (z, hD_z), b^w (z, hD_z)] = \frac{h}{i} \{a, b\}^w (z, hD_z)$ mod $h^3 \Op S^{k+k'-3, l+l'-3}$, where $a \in S^{k, l}$ and $b \in S^{k', l'}$, the following equations correspond to each other in the semiclassical limit:
\begin{align*}
\left\{
\begin{array}{l}
ih\partial_t \gamma (t)= [h^2 P, \gamma (t)], \\
\gamma (0)=  \gamma_0,
\end{array}
\right.
\tag{Q}\label{253171257}
\end{align*}
\begin{align*}
\left\{
\begin{array}{l}
\partial_t f (t, z, \zeta)+ H_p f(t, z, \zeta)=0, \\
f (0, z, \zeta)=  f_0 (z, \zeta),
\end{array}
\right.
\tag{C}\label{253171256}
\end{align*}
where $P=-\partial_{i} g^{ij} (z) \partial_{j}$ is as above. Here, for $k, l \in \R$, $S^{k, l} := S \left(\langle \zeta \rangle^{k} \langle z \rangle^{l}, \frac{dz^2}{\langle z \rangle^2} + \frac{d\zeta^2}{\langle \zeta \rangle^2} \right)$ is the scattering symbol class denoted by H\"ormander's notation (see \cite{Hor}) and its quantization is defined as usual (\cite{Zw}):
\begin{align*}
a^w (z, hD_z)u(z) = \frac{1}{(2\pi h)^d} \int e^{\frac{i}{h}(z-z')\zeta} a \left(\frac{z+z'}{2}, \zeta \right) u(z') dz' d\zeta.
\end{align*}
In (\ref{253171256}), $p(z, \zeta) = g^{ij} (z) \zeta_{i} \zeta_{j}$ is the kinetic energy and $H_p =\frac{\partial p}{\partial \zeta} \frac{\partial}{\partial z} - \frac{\partial p}{\partial z} \frac{\partial}{\partial \zeta}$ is the Hamilton vector field. As a result it would be natural to consider that some smoothing estimate for the transport equation (\ref{253171256}) should be obtained from the orthonormal Strichartz estimates in the semiclassical limit, which are smoothing estimates for the Heisenberg equation (\ref{253171257}). This is known to be true for $P=-\Delta$, i.e. $g^{ij} = \delta_{ij}$. In \cite{Sab, BHLNS} the Strichartz estimates for the transport equation, equivalently the velocity average estimates, are deduced when $P=-\Delta$. Such results are extended to the fractional Schr\"odinger equation ($P= |D_z|^{\alpha}$) and the Klein-Gordon equation ($P= \langle D_z \rangle$) in \cite{BLN}. However there seems to be no result if $P$ is not a Fourier multiplier, especially no result if $g^{ij} \neq \delta_{ij}$. This comes from the lack of the orthonormal Strichartz estimates and an explicit formula of the solution to (\ref{253171257}) which is heavily used when $P=-\Delta, |D_z|^{\alpha}$ or $\langle D_z \rangle$. In Subsection \ref{25371403} we extend the result for $P=-\Delta$ to some pseudodifferential operators including $P=-\Delta_{g}$ where $g$ is a nontrapping scattering metric. Interestingly, as byproducts of such results, we obtain counterexamples of the orthonormal Strichartz estimates caused by trapped sets without constructing quasimodes. This is different from ordinary constructions of counterexamples for the Strichartz estimates.
\subsection{Results on orthonormal Strichartz estimates}\label{25371400}
In this subsection we explain a positive result on the orthonormal Strichartz estimates. First we state the assumptions on manifolds and potentials. A manifold $M^{\circ}$ (or a manifold with boundary $M$) satisfying Assumption \ref{25121723} is called a scattering manifold or a manifold with asymptotically conic ends. 
\begin{assump}\label{25121723}
$(M^{\circ}, g)$ is a $d$-dimensional noncompact complete Riemannian manifold with $d \ge 3$. There exists a compact subset $K \subset M^{\circ}$ such that $M^{\circ} \backslash K$ is diffeomorphic to $(0, \infty) \times Y$. Here $Y$ is a $(d-1)$-dimensional compact connected manifold. We also assume that there exists a compactification $M$ of $M^{\circ}$ such that $\partial M = Y$ and in a collar neighborhood of $\partial M$, $[0, \epsilon_{0})_{x} \times Y_{y}$, $g$ takes a form $g = \frac{dx^2}{x^4} + \frac{h(x)}{x^2}$. Here $h \in C^{\infty} ([0, \epsilon_0); S^2 T^* Y)$.
\end{assump}
We recall that $(M^{\circ}, g)$ is nontrapping if every geodesic $z: \R_t \to M^{\circ}$ goes to $\partial M$ as $t \to \pm \infty$.
\begin{assump}\label{25121738}
$V \in C^{\infty} (M)$ is a real-valued function satisfying $V(x, y) = \mathcal{O} (x^{2+ \epsilon})$ near $\partial M$ for some $\epsilon >0$. 
\end{assump}
A potential $V$ satisfying Assumption \ref{25121738} is said to be of very short-range type. To state our result, we use $L^p (M^{\circ}) = L^p (M^{\circ}, dg) = L^p (M^{\circ}, \sqrt{g} dz)$, where $g= \deter g$. Note that $-\Delta_{g} +V$ is essentially self-adjoint on $C^{\infty} _0 (M^{\circ})$ under Assumptions \ref{25121723} and \ref{25121738}.
\begin{thm}\label{251241323}
Let $(M^{\circ}, g)$ and $V$ be as in Assumptions \ref{25121723} and \ref{25121738}. Suppose $(M^{\circ}, g)$ is nontrapping. We also assume the Schr\"odinger operator $P= -\Delta_{g} +V$ has neither nonpositive eigenvalue nor zero resonance. Then
\begin{align*}
\left\| \sum_{j=0}^{\infty} \nu_j |e^{itP}  f_j|^2 \right\|_{L^{\frac{q}{2}} _t  L^{\frac{r}{2}} _z} \lesssim \|\nu\|_{\ell^{\beta}}
\end{align*}
holds for any orthonormal system $\{f_j\} \subset L^2 (M^{\circ}, dg)$ and any complex-valued sequence $\nu = \{\nu_{j}\}$. Here $\frac{d}{2}$-admissible pair $(q, r)$ and $\beta \in [1, \infty]$ satisfy either of the following conditions: If $r \in [2, \frac{2(d+1)}{d-1})$, then $\beta = \frac{2r}{r+2}$, and if $r \in [\frac{2(d+1)}{d-1}, \frac{2d}{d-2})$, then $\beta < \frac{q}{2}$.
\end{thm}
See Subsection \ref{253191110} for the definition of a $\frac{d}{2}$-admissible pair. The conditions on $(q, r, \beta)$ are exactly the same as in the case of $(M^{\circ}, g) = (\R^d, dz^2 )$ (see \cite{FLLS, FS, BHLNS}). Therefore our result is an extension of such known results to the setting of scattering manifolds. Moreover Theorem \ref{251241323} generalizes the ordinary Strichartz estimates for single function (\cite{HTW, M1, HZ}) to the orthonormal setting. The exponent $\beta$ cannot be taken larger than $\frac{2r}{r+2}$ if $r \in [2, \frac{2(d+1)}{d-1})$ since counterexamples are constructed in \cite{FLLS} for $(\R^d, dz^2 )$. If $r \in [\frac{2(d+1)}{d-1}, \frac{2d}{d-2})$, it is shown in \cite{FS2} that we cannot take $\beta > \frac{q}{2}$ on $(\R^d, dz^2 )$. Even in the orthonormal setting, the presence of conjugate points is irrelevant to the orthonormal Strichartz estimates. The reason is that we only need pointwise estimates for propagators restricted near the diagonal. The proof of Theorem \ref{251241323} is given in Section \ref{241127048} by proving some restriction estimates in the Schatten class for propagators microlocalized near the diagonal. We remark that the uniform Sobolev estimates in the Schatten class are proved in \cite{GHK} in the setting of nontrapping scattering manifolds. Such estimates are refinement of the ordinary uniform Sobolev estimates and used to prove the Lieb-Thirring inequalities for the Schr\"odinger operator with complex potentials, not limited to estimates on individual eigenvalues like the Keller-type bounds. The feasibility of estimating a sum of eigenvalues, which follows from the Schatten estimates for resolvents, corresponds to the extension of the Strichartz estimates to orthonormal systems, which also follows from the Schatten estimates for restriction operators. The result on nontrapping asymptotically hyperbolic manifolds is explained and proved in Subsection \ref{25271436}.
\subsection{Results on quantum and classical correspondence}\label{25371403}
In this subsection we consider quantum and classical correspondence for the Strichartz estimates. Using our main result, counterexamples for the orthonormal Strichartz estimates are also constructed. In the next theorem, our quantum Hamiltonian $P$ is not limited to the Laplace-Beltrami operator also including nonelliptic operators. We say $p \in S^{2, 0}$ is homogeneous of degree $2$ if $p(z, \lambda \zeta) = \lambda^2 p(z, \zeta)$ holds for any $(z, \zeta) \in T^* \R^d$ and $\lambda >0$.
\begin{thm}\label{2412231555}
Assume $p \in S^{2, 0}$ is real-valued, homogeneous of degree $2$, $H_p$ is complete on $T^* \R^d$ and $P = p^w (z, D_z)$ is essentially self-adjoint on $L^2 (\R^d)$ with its core $C^{\infty} _{0} (\R^d)$. If
\begin{align}
\left\| \sum_{j=0}^{\infty} \nu_j |e^{-itP}  f_j|^2 \right\|_{L^{\frac{q}{2}} _t  L^{\frac{r}{2}} _z} \lesssim \|\nu\|_{\ell^{\beta}}\label{2412231605}
\end{align}
holds for any orthonormal system $\{f_j\}$ in $L^2 (\R^d)$, any complex-valued sequences $\nu = \{\nu _j\}$ and some $(q, r, \beta)$ satisfying $q \in [2, \infty]$, $r \in [2, \infty)$, $\frac{2}{q} = d(\frac{1}{2} - \frac{1}{r})$ and $\beta = \frac{2r}{r+2}$, then
\begin{align}
\left\| \int_{\R^d} f \circ e^{-tH_p} (z, \zeta) d\zeta \right\|_{L^{\frac{q}{2}} _t  L^{\frac{r}{2}} _z} \lesssim \|f\|_{L^{\beta} _{z, \zeta}}\label{2412231609}
\end{align}
holds for the same $(q, r, \beta)$.
\end{thm}
Note that $f_0 \circ e^{-tH_p} (z, \zeta)$ is a solution to the transport equation (\ref{253171256}). (\ref{2412231609}) is called the velocity average estimate for the transport equation. Our result is a generalization of \cite{Sab}, where $P=-\Delta$ is considered. See Corollary \ref{251282306} for applications to the Laplace-Beltrami operator on nontrapping scattering manifolds.
\begin{remark}\label{2412231614}
One of the important points in Theorem \ref{2412231555} is that we do not need geometric assumptions except for the completeness of $H_p$ and the essential self-adjointness of $P$, which imply the classical and quantum well-posedness. For example we do not need nontrapping conditions or absence of conjugate points. This is crucial in Section \ref{2412182325}, where counterexamples for the orthonormal Strichartz estimates are proved by using Theorem \ref{2412231555}. Note that if $p$ is elliptic (which means $|p(z, \zeta)| \gtrsim |\zeta|^2$ uniformly in $(z, \zeta) \in T^* \R^d$), the completeness of $H_p$ follows from the conservation of energy. The quantum completeness follows from a standard parametrix construction. If we consider a nonelliptic symbol $p(z, \zeta) = \zeta^2 _{1} + \dots +\zeta^2 _{k} - \zeta^2 _{k+1} - \dots - \zeta^2 _{d}$, the classical completeness follows from an explicit formula of $e^{tH_p}$. The quantum completeness is a consequence of the fact that $P$ is a Fourier multiplier of a real-valued symbol. The orthonormal Strichartz estimates (\ref{2412231605}) follow from the dispersive estimates $|e^{-itP} (z, z')| \lesssim |t|^{-\frac{d}{2}}$ and \cite[Theorem 1.1]{H3}. We leave the case of nonelliptic variable coefficient operators for future work. 
\end{remark}
To state the Strichartz estimates for the transport equation, we define the KT-admissible quadruplet, which appears in the exponent of the Strichartz estimates.
\begin{defn}\label{2412232249}
A quadruplet $(q, r, p, a) \in [1, \infty]^4$ is called a KT-admissible quadruplet if $a= \hm (p, r), \frac{1}{q} = \frac{d}{2} (\frac{1}{p} - \frac{1}{r}), p_{*} (a) \le p \le a \le r \le r_{*} (a)$ and $(q, r, p, d) \neq (a, \infty, \frac{a}{2}, 1)$ hold. Here $\hm (p, r)$ denotes the harmonic mean of $p$ and $r$, i.e. $\hm (p, r)^{-1} = \frac{1}{2} (\frac{1}{p} + \frac{1}{r})$. If $\frac{d+1}{d} \le a \le \infty$, then $(p_{*} (a), r_{*} (a)) = (\frac{da}{d+1}, \frac{da}{d-1})$. If $1 \le a \le \frac{d+1}{d}$, then $(p_{*} (a), r_{*} (a)) = (1, \frac{a}{2-a})$. A KT-admissible quadruplet $(q, r, p, a)$ is called the endpoint if $(q, r, p) = (a, r_{*} (a), p_{*} (a))$ and $\frac{d+1}{d} \le a < \infty$ hold.
\end{defn}
\begin{thm}[Strichartz estimates for transport equations]\label{253201908}
Let $p \in S^{2, 0}$ be as in Theorem \ref{2412231555}. Suppose (\ref{2412231605}) holds for any $(q, r, \beta)$ satisfying $q, r \in [2, \infty], r \in [2, \frac{2(d+1)}{d-1}), \frac{2}{q} = d(\frac{1}{2} - \frac{1}{r})$ and $\beta = \frac{2r}{r+2}$. Then
\begin{align}
\|f \circ e^{-tH_p}\|_{L^q _t L^r _z L^p _{\zeta}} \lesssim \|f\|_{L^a _{z, \zeta}}\label{24122323040}
\end{align}
holds for any non-endpoint KT-admissible quadruplet $(q, r, p, a)$. Furthermore if $(q, r, p, a)$ and $(\tilde{q}, \tilde{r}, \tilde{p}, a')$ are non-endpoint KT-admissible quadruplets, then
\begin{align}
\left\| \int_{0}^{t} F(s) \circ e^{-(t-s)H_p} (z, \zeta) ds \right\|_{L^q _t L^r _z L^p _{\zeta}} \lesssim \|F\|_{L^{\tilde{q}'} _t L^{\tilde{r}'} _z L^{\tilde{p}'} _{\zeta}} \label{24122323550}
\end{align}
holds.
\end{thm}
In the simplest case, $p(z, \zeta) = |\zeta|^2$, (\ref{24122323040}) fails at the endpoint (see \cite{BBGL}). Moreover by \cite{O1}, (\ref{24122323040}) holds if and only if $(q, r, p, a)$ is a non-endpoint KT-admissible quadruplet. We also refer to \cite{BLNS, CP, GP, HCFH, KT, O2} for the dispersive or Strichartz estimates when $p(z, \zeta) = |\zeta|^2$. On the other hand there are few results for variable coefficient operators or on manifolds. In $1$-dimensional case, some weighted Strichartz estimates are proved for $p(z, \zeta) = g(z) \zeta^2$ with $g \sim 1$ in \cite{Sal1}. In higher dimensions, if $p(z, \zeta) = g^{ij} (z) \zeta_{i} \zeta_{j}$, $g^{ij}$ is a compactly supported perturbation of $\delta_{ij}$ and $e^{tH_p}$ is nontrapping, (\ref{24122323040}) is proved in \cite{Sal2}. However if $g^{ij}$ is a long-range perturbation of $\delta_{ij}$ and $e^{tH_p}$ is nontrapping, the estimates obtained in \cite{Sal2} are local-in-time and have derivative loss. The following result is a refinement of such estimates in the setting of scattering metrics, including long-range perturbations.
\begin{corollary}\label{253202133}
Let $g$ be a nontrapping scattering metric on $\R^d$ with $d \ge 3$, which satisfies $|\partial^{\alpha} _z g_{ij} (z)| \lesssim \langle z \rangle^{-|\alpha|}$ for any $1 \le i, j \le d$ and $\alpha \in \N^d _0$. Set $p(z, \zeta) = g^{ij} (z) \zeta_{i} \zeta_{j} \in S^{2, 0}$. Then (\ref{24122323040}) and (\ref{24122323550}) hold for any non-endpoint KT-admissible quadruplets $(q, r, p, a)$ and $(\tilde{q}, \tilde{r}, \tilde{p}, a')$.
\end{corollary}
We remark that \cite{BF} considers the Strichartz estimates for the sub-Laplacian on the Heisenberg group. For special initial data, solutions to the Schr\"odinger equation behave like those to a transport equation. Hence \cite[Theorem 1.1]{BF} contains a Strichartz-type estimate for a transport equation with special initial data. \cite{VRVR} considers pointwise decay estimates in time for the velocity average of solutions to the transport equation on $2$D nontrapping asymptotically hyperbolic manifolds. Since their results do not contain decay in space, it seems difficult to derive the Strichartz estimates. Our proof of Corollary \ref{253202133} and the orthonormal Strichartz estimates Theorem \ref{252101626} may apply to the nontrapping asymptotically hyperbolic setting. Now we consider counterexamples for the orthonormal Strichartz estimates caused by trapped sets.
\begin{defn}\label{253211014}
We say that the sharp orthonormal Strichartz estimates fail if and only if for any $(q, r, \beta)$ satisfying $q, r \in [2, \infty), r \in [2, \frac{2(d+1)}{d-1}), \frac{2}{q} = d(\frac{1}{2} - \frac{1}{r})$ and $\beta = \frac{2r}{r+2}$, (\ref{2412231605})
does not hold uniformly in orthonormal $\{f_j\} \subset L^2 (\R^d)$ and $\nu = \{\nu_{j}\}$.
\end{defn}
\begin{thm}\label{25371352}
Let $p \in S^{2, 0}$ be real-valued, homogeneous of degree $2$, $H_p$ be complete on $T^* \R^d$ and $P = p^w (z, D_z)$ be essentially self-adjoint on $L^2 (\R^d)$ with its core $C^{\infty} _{0} (\R^d)$.

\noindent (i) Assume $d=1$. If there exists a periodic trajectory $\gamma \subset T^* \R$ associated to $H_p$, the sharp orthonormal Strichartz estimates fail for $P$.

\noindent (ii) If there exists a periodic stable trajectory $\gamma \subset T^* \R^d$ associated to $H_p$, the sharp orthonormal Strichartz estimates fail for $P$.

\noindent (iii) There exists a Riemannian metric $g$ on $\R^d$ such that $g = dz^2$ outside a compact set and the sharp orthonormal Strichartz estimates fail for $P=-\Delta_{g}$.
\end{thm}
For the definition of the stability of a periodic trajectory, see Section \ref{2412182325}. An interesting feature in the proof of Theorem \ref{25371352} is that we do not need any quasimode. Our strategy is to use Theorem \ref{2412231555}. The point is that by the stability condition we can take an initial state $f \in C^{\infty} _0 (T^* \R^d)$ in (\ref{2412231609}) concentrating on the trapped set. Concerning (iii) we construct desired metrics $g$ using periodic geodesics on the sphere and our construction is explicit (see Proposition \ref{251151538}). The proof of Theorem \ref{2412231555}, \ref{253201908} and Corollary \ref{253202133} are given in Section \ref{2412182324}. Theorem \ref{25371352} is proved in Section \ref{2412182325}.
\subsection{Applications to nonlinear equations}\label{25371533}
It is known that the orthonormal Strichartz estimates are useful to prove the well-posedness or scattering for infinite fermionic systems. See references in \cite{H3} for such results and \cite{AKN1, AKN2, LeSa, LaSa, Sm} for the semiclassical limit of such equations. Theorems \ref{251241323} and \ref{252101626} are also applicable but we omit details here since arguments are identical to \cite{H1, H2, H3}. It is also notable that Theorem \ref{251241323} with the Littlewood-Paley theorem (\cite[Proposition 2.2]{Zh1}) yields the refined Strichartz estimates:
\begin{align}
\|e^{it\Delta_{g}} u_0\|_{L^q _t L^r _z} \lesssim \|u_0\|_{\dot{B}^{0} _{2, 2\beta}}, \label{25381740}
\end{align}
where $(q, r, \beta)$ is as in Theorem \ref{251241323} and $\|f\|_{\dot{B}^{0} _{2, 2\beta}} :=\| \{\|\phi_{j} (\sqrt{-\Delta_{g}}) f\|_2 \} \|_{\ell^{2\beta}}$ with a homogeneous Littlewood-Paley decomposition $\{\phi_{j}\}$. The proof of (\ref{25381740}) is identical to \cite{FS, H1, H2}. Using (\ref{25381740}) we can refine a small-data scattering for the $L^2$-critical NLS considered in \cite{BM}, that is, we can show that for any $M>0$ there exists $\epsilon >0$ such that the $L^2$-critical NLS has a unique global scattering solution provided initial data $u_0$ satisfies $\|u_0\|_2 < M$ and $\|u_0\|_{\dot{B}^{0} _{2, \infty}} < \epsilon$. We refer to \cite{H3} for the proof, where the Aharonov-Bohm Hamiltonian is considered.

In this subsection we consider the Boltzmann equation:
 \[
\left\{
\begin{array}{l}
\partial _{t} f(t, z, \zeta) +H_p f(t, z, \zeta) = Q(f, f) (t, z, \zeta) \\
f(0, z, \zeta) = f_0 (z, \zeta) \\
\end{array}
\right.
\tag{B}\label{241224048}
\]    
on $T^* \R^d$. Here we assume $p \in S^{2, 0}$ is real-valued, homogeneous of degree $2$, $H_p$ is complete on $T^* \R^d$ and the Strichartz estimates (\ref{24122323040}) and (\ref{24122323550}) are satisfied. Typical examples are $p(z, \zeta) = \sum_{i, j = 1}^{d} g^{ij} (z) \zeta _{i} \zeta _{j}$ for a nontrapping scattering metric $g = (g_{ij})$ (see Corollary \ref{253202133}). The nonlinearity (collision term) in the right hand side of (\ref{241224048}) is given by
\begin{align*}
Q(f, f) (t, z, \zeta) = \int_{\R^d} \int_{\mathbb{S}^{d-1}} (f' f' _{*} -f f_{*}) B(\zeta - \zeta _{*}, \omega) d\omega d \zeta _{*},
\end{align*}
where $f' = f (t, z, {\zeta}')$, $f' _{*} = f(t, z, {\zeta}' _{*})$, $f_{*} = f(t, z, \zeta _{*})$ and the relations of pre-collisional and post-collisional momentum are given by $\zeta ' = \zeta - [\omega \cdot (\zeta - \zeta _{*})] \omega$, $\zeta ' _{*} = \zeta _{*} + [\omega \cdot (\zeta - \zeta _{*})] \omega$. Concerning the collisional kernel $B$, we assume $B$ is a cut-off soft potential. Precisely we assume that $B(\zeta - \zeta _{*}, \omega) = |\zeta - \zeta _{*}|^{\gamma} b (\cos \theta)$ for some $\gamma \in (-d, 0]$ and $ \cos \theta = \frac{(\zeta - \zeta _{*}) \cdot \omega}{|\zeta - \zeta _{*}|}$. Here $b$ is a nonnegative measurable function supported in $\{ \cos \theta \ge 0\}$ and satisfies Grad's cut-off assumption $0 \le \int_{\mathbb{S}^{d-1}} b(\cos \theta) d \omega < \infty$.
To introduce supplementary function spaces we set $\Lambda = \left\{ (q, r, p) \in [1, \infty]^3 \mid \frac{1}{q} = \frac{d}{p} -1, \frac{1}{r} = \frac{2}{d} - \frac{1}{p}, \frac{1}{d} < \frac{1}{p} < \frac{d+1}{d^2} \right\}$. In the main theorem we consider a small-data scattering for the $\gamma = -1$ model (high temperature situation). 
\begin{thm}\label{25121757}
Assume $d =3$, $\gamma = -1$, (\ref{24122323040}) and (\ref{24122323550}) for any non-endpoint KT-admissible quadruplets $(q, r, p, a)$ and $(\tilde{q}, \tilde{r}, \tilde{p}, a')$. If $f_0 \in L^3 \cap L^{\frac{15}{8}} _{z, \zeta}$ satisfies $f_0 \ge 0$ and $\|f_0\|_{L^3 \cap L^{\frac{15}{8}} _{z, \zeta}}$ is sufficiently small, then (\ref{241224048}) has a unique nonnegative solution $f \in C([0, \infty); L^3 _{z, \zeta}) \cap L^{\boldsymbol{q}} ([0, \infty); L^{\boldsymbol{r}} _z L^{\boldsymbol{p}} _{\zeta}) \cap L^2 ([0, \infty); L^{\frac{30}{11}} _{z} L^{\frac{10}{7}} _{\zeta})$ for any $(\boldsymbol{q}, \boldsymbol{r}, \boldsymbol{p}) \in \Lambda$. Moreover there exists $f_{\infty} \in L^3 _{z, \zeta}$ such that
\begin{align}
\|f(t) - f_{\infty} \circ e^{-tH_p} \|_{L^3 _{z, \zeta}} \to 0 \quad as \quad t \to \infty. \label{25121807}
\end{align}
\end{thm}
If $p(z, \zeta) = |\zeta|^2$, (\ref{241224048}) is the ordinary Boltzmann equation. A small-data scattering is proved in \cite{CDP, HJ2} with $d=2, 3$. Our result is an extension to the setting of nontrapping scattering metrics. To the best knowledge of the author, there seems to be no other result for the scattering on noncompact manifolds except for $\R^d$. See \cite{CHR}, \cite{DILS} or \cite{San} for the nonlinear Vlasov equation on compact Anosov manifolds, for the Vlasov-Poisson system on $\mathbb{H}^2$ and $\mathbb{S}^2$ or for some decay estimates for the Boltzmann equation on compact manifolds. We remark that \cite{CheHo} derives the (quantum) Boltzmann equation from a many-body system in the mean-field limit. Our proof of Theorem \ref{25121757} may have similarity since it relies on the Strichartz estimates, which follow from the orthonormal Strichartz estimates, that is, estimates for infinitely many-body systems. See Section \ref{2412182326} for the proof and comments on the condition $f_0 \ge 0$.
\subsection{Notations}\label{253191110}
For $(q, r) \in [2, \infty]^2$ and $\sigma >0$, we say that $(q, r)$ is a $\sigma$-admissible pair if $\frac{2}{q} =2\sigma (\frac{1}{2} - \frac{1}{r})$ and $(\sigma, q, r) \ne (1, 2, \infty)$ hold. For $z \in \C$ and $a \in \C \setminus i(-\infty, 0]$ we define $a^z = \exp (z \log a)$, where $\log$ is a branch defined on $\C \setminus i(-\infty, 0]$ and satisfies $\arg \log r =0$ and $\arg \log (-r) = \pi$ for $r >0$. For $p, q \in [1, \infty]$, $\| \cdot \|_{p \to q}$ denotes the operator norm from $L^p$ to $L^q$. For a tempered distribution $u$, $\mathcal{F} [u]$ stands for its Fourier transform.

\section{Positive results for orthonormal Strichartz estimates on manifolds with ends}\label{241127048}
In this section we prove the orthonormal Strichartz estimates on nontrapping scattering manifolds ({\S}\ref{25271431}) and on nontrapping asymptotically hyperbolic manifolds ({\S}\ref{25271436}). Throughout this section $z$ denotes local coordinates away from $\partial M$. The Laplace-Beltrami operator $\Delta_{g}$ is locally expressed as $\frac{1}{\sqrt{g}} \partial_{i} g^{ij} \sqrt{g} \partial_{j}$, where $(g^{ij}) = (g_{ij})^{-1}$ and $g = g_{ij} dz^i dz^j$. Near the boundary we use local coordinates $(x, y)$ as in the definition. The proof given in {\S}\ref{25271431} is analogous to the abstract results in \cite{H3, FMSW} (see also \cite{Ng}) since microlocalized propagators satisfy similar bounds globally in time. Contrary to this we need different arguments in {\S}\ref{25271436}. This is because microlocalized propagators at low energy do not contain sufficient decay in time due to the totally different geometry of asymptotically hyperbolic spaces. In order to compensate for this, we need to use exponential decay in space as in Lemma \ref{25251030}.  
\subsection{Nontrapping scattering manifolds}\label{25271431}
For the proof of Theorem \ref{251241323} we use a decomposition of the propagator as in \cite{HZ}. We assume, for a while, $V =0$ hence $P= -\Delta_{g}$. It is shown in \cite{HZ} that there exists an energy-dependent operator partition of unity $\{Q_{j} (\lambda)\}_{j=1}^{N}$ on $L^2 (M^{\circ})$ such that
\begin{align*}
&\Iden = \sum_{j=1}^{N} Q_{j} (\lambda) \quad for \; all \; \lambda \in [0, \infty), \\
& U_{j} (t) = \int_{0}^{\infty} e^{it\lambda ^{2}} Q_{j} (\lambda) dE_{\sqrt{P}} (\lambda)
\end{align*} 
is a well-defined operator on $L^2 (M^{\circ})$. Here $N$ is independent of $\lambda \in [0, \infty)$ and $\|U_{j} (t) \|_{2 \to 2}$ are uniformly bounded in $t \in \R$. They satisfy $e^{itP} = \sum_{j=1}^{N} U_{j} (t)$. The important point is that the propagators $U_{j} (t)U_{j} (s)^*$ are microlocalized by $\{Q_{j} (\lambda)\}_{j=1}^{N}$ near the diagonal so that the following dispersive estimates hold (i.e. conjugate points have no effect for these propagators):
\begin{align}
\left| \int_{0}^{\infty} e^{it\lambda ^{2}} (Q_{j} (\lambda) dE_{\sqrt{P}} (\lambda) Q_j (\lambda)^*) (z, z')d\lambda \right| \lesssim |t|^{-\frac{d}{2}} \label{251241554}.
\end{align}
In the next proposition we show the orthonormal Strichartz estimates for these microlocalized propagators using (\ref{251241554}).
\begin{proposition}\label{251241557}
Let $P = -\Delta_{g}$ be the Laplace-Beltrami operator on $(M^{\circ}, g)$ satisfying Assumption \ref{25121723} and nontrapping condition. Then, for any $\frac{d}{2}$-admissible pair $(q, r)$ satisfying $1+d < \tilde{r} = 2(\frac{r}{2})' < 2+d$, $\beta = \frac{2r}{r+2}$ and $\bullet \in \{1, \dots, N\}$, we have
\begin{align}
\left\| \sum_{j=0}^{\infty} \nu_j |U_{\bullet} (t) f_j|^2 \right\|_{L^{\frac{q}{2}, \beta} _t  L^{\frac{r}{2}} _z} \lesssim \|\nu\|_{\ell^{\beta}}. \label{251251602}
\end{align}
\end{proposition}
\begin{proof}
We define, for $\omega \in \{ \omega \in \C \mid \re \omega \in [-1, \frac{d}{2}]\}$, $\epsilon \in (0, 1)$ and a simple function $F$,
\begin{align*}
T_{\omega, \epsilon} F(t, z) = \int_{\R} \chi_{\{\epsilon < |t-s|< \frac{1}{\epsilon}\}} (t-s)^{\omega} U_{\bullet} (t)U_{\bullet} (s)^* F(s)ds
\end{align*}
and $TF(t, z) = \int_{\R} U_{\bullet} (t)U_{\bullet} (s)^* F(s)ds$. If $\re \omega \in (\frac{d-1}{2}, \frac{d}{2})$ and $W_1, W_2$ are simple functions, by (\ref{251241554}) (note that the left hand side of (\ref{251241554}) is the integral kernel of $U_{\bullet} (t)U_{\bullet} (s)^*$) we have
\begin{align}
\|W_1 T_{\omega, \epsilon} W_2\|^2 _{\mathfrak{S}^2} &\lesssim \int_{\epsilon < |t-s|< \frac{1}{\epsilon}} |t-s|^{2\re \omega -d} |W_1 (t, z)|^2 |W_2 (s, z')|^2 dsdtdzdz' \notag \\
& \lesssim \int |t-s|^{2\re \omega -d} \|W_1 (t)\|^2 _2 \|W_2 (s)\|^2 _2 dsdt \notag  \\
& \lesssim \|W_1\|^2 _{L^{2u, 4} _t L^{2} _z} \|W_2\|^2 _{L^{2u, 4} _t L^{2} _z} \label{251251148}
\end{align}
with an implicit constant at most exponential in $\im \omega$. Here the third line follows from the Hardy-Littlewood-Sobolev inequality and $\frac{2}{u} = 2+2\re \omega -d$. Next if $\re \omega =-1$,
\begin{align*}
\|T_{\omega, \epsilon} F(t)\|_2 &= \left\| U_{\bullet} (t) \int_{\R} \chi_{\{\epsilon < |t-s|< \frac{1}{\epsilon}\}} (t-s)^{\omega} U_{\bullet} (s)^* F(s)ds \right\|_2 \\
& \lesssim \left\| \int_{\R} \chi_{\{\epsilon < |t-s|< \frac{1}{\epsilon}\}} (t-s)^{\omega} U_{\bullet} (s)^* F(s)ds \right\|_2
\end{align*}
and the Fourier transform with respect to $t$-variable yields
\begin{align*}
\|T_{\omega, \epsilon} F\|_{L^2 _t L^2 _z} &\lesssim \|(t^{\omega} \chi_{\{\epsilon < |t|< \frac{1}{\epsilon}\}}) * (U_{\bullet} (t)^* F(t)) \|_{L^2 _t L^2 _z} \\
& = \|\mathcal{F} [t^{\omega} \chi_{\{\epsilon < |t|< \frac{1}{\epsilon}\}}] \mathcal{F} [U_{\bullet} (t)^* F(t)]\|_{L^2 _t L^2 _z} \\
& \lesssim \langle \im \omega \rangle e^{\pi |\im \omega|} \|U_{\bullet} (t)^* F(t)\|_{L^2 _t L^2 _z} \lesssim \langle \im \omega \rangle e^{\pi |\im \omega|} \|F\|_{L^2 _t L^2 _z}.
\end{align*}
Here the bound $\|\mathcal{F} [t^{\omega} \chi_{\{\epsilon < |t|< \frac{1}{\epsilon}\}}]\|_{\infty} \lesssim \langle \im \omega \rangle e^{\pi |\im \omega|}$ uniformly in $\epsilon \in (0, 1)$ is used. We give its proof in Appendix \ref{251252117}. Thus we obtain $\|W_1 T_{\omega, \epsilon} W_2\|_{\mathfrak{S}^{\infty}} \lesssim \langle \im \omega \rangle e^{\pi |\im \omega|} \|W_1\|_{L^{\infty} _t L^{\infty} _z} \|W_2\|_{L^{\infty} _t L^{\infty} _z}$. Now combining this with (\ref{251251148}) and using the complex interpolation, we have
\begin{align}
\|W_1 T_{0, \epsilon} W_2\|_{\mathfrak{S}^{\tilde{r}}} \lesssim \|W_1\|_{L^{\tilde{q}, 2\tilde{r}} _t L^{\tilde{r}} _z} \|W_2\|_{L^{\tilde{q}, 2\tilde{r}} _t L^{\tilde{r}} _z}. \label{251251150}
\end{align}
Next we show $\|W_1 T_{0, \epsilon} W_2 - W_1 T W_2\|_{\mathcal{L} (L^2 (M^{\circ}))} \to 0$ as $\epsilon \to 0$. We split the operator into
\begin{align*}
(W_1 T_{0, \epsilon} W_2 - W_1 T W_2)F(t, z) &= W_1 (t, z) \int_{\{|t-s| < \epsilon \}} U_{\bullet} (t)U_{\bullet} (s)^* W_2 (s)F(s) ds \\
& \quad + W_1 (t, z) \int_{\{|t-s| > \frac{1}{\epsilon} \}} U_{\bullet} (t)U_{\bullet} (s)^* W_2 (s)F(s) ds = I + I\hspace{-1.2pt}I.
\end{align*}
Since $W_1$ and $W_2$ are simple functions we may assume $\supp W_1, \supp W_2 \subset [-T, T] \times \Omega \Subset \R \times M^{\circ}$. For $I$, by H\"older's inequality, we have
\begin{align*}
\|I\|_{L^2 _t L^2 _z} = \|I\|_{L^2 ([-T, T] \times \Omega)} \lesssim \left\| \int_{\{|t-s| < \epsilon \}} \|F(s)\|_2 ds \right\|_{L^2 ([-T, T])} \lesssim \epsilon ^{\frac{1}{2}} \|F\|_{L^2 _t L^2 _z}.
\end{align*}
For $I\hspace{-1.2pt}I$, using (\ref{251241554}), we obtain
\begin{align*}
\|I\hspace{-1.2pt}I\|_{L^2 _t L^2 _z} &\lesssim \left\| \int_{\{|t-s| > \frac{1}{\epsilon} \}} \|U_{\bullet} (t)U_{\bullet} (s)^* W_2 (s)F(s)\|_{\infty} ds \right\|_{L^2 ([-T, T])} \\
& \lesssim \left\| \int_{\{|t-s| > \frac{1}{\epsilon} \}} \epsilon ^{\frac{d}{2}} \|W_2 (s)\|_2 \|F(s)\|_2 ds \right\|_{L^2 ([-T, T])} \lesssim \epsilon ^{\frac{d}{2}} \|F\|_{L^2 _t L^2 _z}.
\end{align*}
In the above estimates, implicit constants depend on $T, |\Omega|$ (and hence $W_1$ and $W_2$) but uniform in $\epsilon$. Therefore we get
\begin{align}
\|W_1 T_{0, \epsilon} W_2 - W_1 T W_2\|_{\mathcal{L} (L^2 (M^{\circ}))} \lesssim \epsilon ^{\frac{1}{2}} + \epsilon ^{\frac{d}{2}} \to 0 \label{251251531}
\end{align}
as $\epsilon \to 0$. We denote the singular values of a compact operator $A$ by $\{\mu _n (A)\}$. Then (\ref{251251531}) yields $\mu _n (W_1 T_{0, \epsilon} W_2) \to \mu _n (W_1 T W_2)$ as $\epsilon \to 0$ for all $n \in \N$ by a discussion in \cite{Si} p.26. By Fatou's lemma and (\ref{251251150}) we obtain $W_1 T W_2 \in \mathfrak{S}^{\tilde{r}}$ and
\begin{align}
\|W_1 TW_2\|_{\mathfrak{S}^{\tilde{r}}} = \|\{\mu _n (W_1 T W_2)\}\|_{\ell^{{\tilde{r}}}} &\le \varliminf_{\epsilon \to 0} \|\{\mu _n (W_1 T_{0, \epsilon} W_2)\}\|_{\ell^{{\tilde{r}}}} \notag \\
&= \varliminf_{\epsilon \to 0} \|W_1 T_{0, \epsilon} W_2\|_{\mathfrak{S}^{\tilde{r}}}
 \lesssim \|W_1\|_{L^{\tilde{q}, 2\tilde{r}} _t L^{\tilde{r}} _z} \|W_2\|_{L^{\tilde{q}, 2\tilde{r}} _t L^{\tilde{r}} _z}. \label{251251556}
\end{align} 
Now (\ref{251251602}) follows from (\ref{251251556}) and the duality principle (see \cite{H3} Lemma 2.1 or \cite{BLN} Proposition 1 for our situation but they are originated in \cite{FS}). 
\end{proof}
To add a potential $V$, we employ the perturbation method in \cite{H1}. The assumptions on $V$ (very short-range condition, absence of nonpositive eigenvalues and zero resonances) are made only to use the uniform resolvent estimates. Actually any potential $V$ is allowed as long as $|V|^{\frac{1}{2}}$ is $-\Delta_{g}$-smooth and $|V|^{\frac{1}{2}} P_{\ac}$ is $-\Delta _{g} +V$-smooth. 
\begin{proof}[Proof of Theorem \ref{251241323}]
First we assume $V =0$. Then by summing up (\ref{251251602}) with respect to $j = 1, \dots, N$, we obtain
\begin{align*}
\left\| \sum_{j=0}^{\infty} \nu_j |e^{itP} f_j|^2 \right\|_{L^{\frac{q}{2}, \beta} _t  L^{\frac{r}{2}} _z} \lesssim \|\nu\|_{\ell^{\beta}}
\end{align*}
for all $(q, r)$ and $\beta$ as in Proposition \ref{251241557}. Then Theorem \ref{251241323} follows from a simple interpolation argument (see Proof of Theorem 1.1 in \cite{H3}). Next we consider general cases. We use the following uniform resolvent estimates (Proposition 3.1 in \cite{ZZ2}):
\begin{align*}
\sup_{\sigma \in \C \setminus \R} \|\langle z \rangle^{-1- \frac{\epsilon}{2}} (P- \sigma)^{-1} \langle z \rangle^{-1-\frac{\epsilon}{2}} \|_{\mathcal{L} (L^2 (M^{\circ}))} < \infty.
\end{align*} 
We decompose $P= -\Delta _{g} +V = -\Delta _{g} + |V|^{\frac{1}{2}} \cdot (|V|^{\frac{1}{2}} \sgn V)$ and use Theorem 2.3 in \cite{H1} to obtain the desired estimates.
\end{proof}
\begin{remark}\label{251261538}
Even if $(M^{\circ}, g)$ has a mild trapped set in the sense of \cite{BGH}, we have a similar estimate for the propagator at low energy $e^{itP} \phi (P)$, where $\phi$ is a cutoff function around low energy. This is because trapped sets are irrelevant to low energy estimates and we can use an analogous decomposition of the propagator. Contrary to this, high energy estimates do not seem to be easy because Littlewood-Paley type arguments are not effective in the orthonormal setting (see Introduction of \cite{BHLNS}), though they are used in the ordinary Strichartz estimates (see \cite{ZZ2} and \cite{BM}). 
\end{remark}
\subsection{Nontrapping asymptotically hyperbolic manifolds}\label{25271436}
In this subsection we consider the orthonormal Strichartz estimates on nontrapping asymptotically hyperbolic manifolds. First we recall their definition in a similar manner as in Assumption \ref{25121723}.
\begin{assump}\label{251241320}
$(M^{\circ}, g)$ is a $d$-dimensional noncompact complete Riemannian manifold with $d \ge 3$. There exists a compact subset $K \subset M^{\circ}$ such that $M^{\circ} \backslash K$ is diffeomorphic to $(0, \infty) \times Y$. Here $Y$ is a $(d-1)$-dimensional compact connected manifold. We also assume that there exists a compactification $M$ of $M^{\circ}$ such that $\partial M = Y$ and in a collar neighborhood of $\partial M$, $[0, \epsilon_{0})_{x} \times Y_{y}$, $g$ takes a form $g = \frac{dx^2}{x^2} + \frac{h(x)}{x^2}$. Here $h \in C^{\infty} ([0, \epsilon_0); S^2 T^* Y)$. 
\end{assump}
Under Assumption \ref{251241320}, $-\Delta_{g} \upharpoonright_{C^{\infty} _0 (M^{\circ})}$ is essentially self-adjoint on $L^2 (M^{\circ}, dg)$. Its spectrum satisfies $\sigma (-\Delta_{g}) = \sigma_{\ac} (-\Delta_{g}) \cup \sigma_{\pp} (-\Delta_{g})$, $\sigma_{\ac} (-\Delta_{g}) = \left[\frac{(d-1)^2}{4}, \infty \right)$ and $\sigma_{\pp} (-\Delta_{g}) \subset \left(0, \frac{(d-1)^2}{4} \right)$.
Contrary to the result on asymptotically conic manifolds, the orthonormal Strichartz estimates on asymptotically hyperbolic manifolds admit wider range of exponents $(q, r, \beta)$. One reason of this is the Kunze-Stein phenomenon, which allows Young-type inequalities with wider admissible exponents than in the Euclidean space.
\begin{thm}\label{252101626}
Let $(q, r) \in [2, \infty] \times (2, \infty]$ satisfy $\frac{2}{q} \ge d(\frac{1}{2} - \frac{1}{r})$ and $\beta \in [1, \infty]$. Suppose $(M^{\circ}, g)$ is nontrapping. We assume either of the following conditions: $(i)$ $\frac{2}{q} > d(\frac{1}{2} - \frac{1}{r})$, $\frac{1}{r} > \frac{2}{q} - \frac{1}{2}$ and $\beta < \frac{2r}{r+2}$, $(ii)$ $\frac{2}{q} = d(\frac{1}{2} - \frac{1}{r})$, $\frac{1}{r} > \frac{2}{q} - \frac{1}{2}$ and $\beta = \frac{2r}{r+2}$, $(iii)$ $\frac{1}{r} \le \frac{2}{q} - \frac{1}{2}$ and $\beta < \frac{q}{2}$. Then
\begin{align}
\left\| \sum_{j=0}^{\infty} \nu_j |e^{it\Delta_{g}} P_{\ccc}  f_j|^2 \right\|_{L^{\frac{q}{2}} _t  L^{\frac{r}{2}} _z} \lesssim \|\nu\|_{\ell^{\beta}} \label{252101709}
\end{align}
holds for any orthonormal system $\{f_j\}_{j=0}^{\infty} \subset L^2 (M^{\circ}, dg)$ and any complex-valued sequence $\nu = \{\nu_{j}\}_{j=0}^{\infty}$.
\end{thm}
Let $P= \sqrt{(-\Delta_g - \frac{(d-1)^2}{4})_+}$. We take cutoff functions $\chi_{\low}$ and $\chi_{\infty}$ such that $\chi_{\low} \in C^{\infty} ([0, \infty); [0, 1])$, $\supp \chi_{\low} \subset [0, 2]$, $\chi_{\infty} \in C^{\infty} ([0, \infty); [0, 1])$, $\supp \chi_{\infty} \subset [1, \infty)$ and $\chi_{\low} + \chi_{\infty} = 1$ in $[0, \infty)$. Then, letting $P_{\ccc}$ be the orthogonal projection onto the (absolutely) continuous subspace of $-\Delta_{g}$, it is shown in \cite{Che1} that
\begin{align*}
e^{-it\Delta_{g} - \frac{it(d-1)^2}{4}} P_{\ccc} = \int_{0}^{\infty} e^{it\lambda ^2} dE_{P} (\lambda ) = U_{\low} (t) + \sum_{k=0}^{N} U_{k} (t)
\end{align*}
holds, where the propagators in the last two terms are defined by
\begin{align*}
U_{\low} (t) = \int_{0}^{\infty} e^{it\lambda ^2} \chi_{\low} (\lambda) dE_{P} (\lambda ),\quad U_{k} (t) = \int_{0}^{\infty} e^{it\lambda ^2} \chi_{\infty} (\lambda) Q_{k} (\lambda) dE_{P} (\lambda ).
\end{align*}
Here $\{Q_{k} (\lambda)\}_{k=0}^{N}$ is again a pseudodifferential partition of unity: $\sum_{k=0}^{N} Q_k (\lambda) = \Iden$. $N$ is independent of $\lambda \in [0, \infty)$. By definition $U_{\low} (t)$ is uniformly bounded on $L^2 (M^{\circ})$ and $U_{k} (t)$ are also well-defined uniformly bounded operators. In asymptotically hyperbolic manifolds, pointwise estimates of these propagators are different from those in asymptotically conic manifolds. In the present situation the propagators have exponential decay in the spatial variable and the decay order in time differs depending on the energy region. We set
\begin{align}
&K_{k} (t, z, z') = \int_{0}^{\infty} e^{it\lambda ^2} \chi^2 _{\infty} (\lambda ) (Q_k (\lambda)dE_{P} (\lambda ) Q^* _k (\lambda)) (z, z') d\lambda, \label{25251015} \\
& K_{\low} (t, z, z') = \int_{0}^{\infty} e^{it\lambda ^2} \chi^2 _{\low} (\lambda) dE_P (\lambda) (z, z') d\lambda. \label{25251017}
\end{align}
Then by Proposition 6 in \cite{Che1} the following pointwise estimates hold: If $|t| > 1+d(z, z')$ then $|K_k (t, z, z')| \lesssim |t|^{-\infty} e^{-\frac{d-1}{2} d(z, z')}$ holds. If $|t| < 1+d(z, z')$ then $|K_k (t, z, z')| \lesssim |t|^{-\frac{d}{2}} (1+d(z, z'))^{\frac{d-1}{2}} e^{-\frac{d-1}{2} d(z, z')}$ holds. For all $t \in \R$, $|K_{\low} (t, z, z')| \lesssim |t|^{-\frac{3}{2}} (1+d(z, z')) e^{-\frac{d-1}{2} d(z, z')}$ holds. Based on these pointwise estimates we prove the following lemma which is used to estimate the Schatten norms later. Note that the estimates in \cite[{\S}8]{Che1} are not sufficient since we need to estimate the Hilbert-Schmidt norm. We use $d(\cdot, \cdot)$ or $d_Y (\cdot, \cdot)$ to denote the distance function on $M^{\circ}$ or $Y$ respectively. $M^2 _0$ is a manifold with corners (called the $0$-double space) obtained by blowing up $\{(0, y, 0, y) \in M^2 \mid y \in Y\}$ in $M^2$. See Section 3 of \cite{CheHa} for more details about this blown-up space.
\begin{lemma}\label{25251030}
Let $J_{\bullet} (t, z, z') = |K_{\bullet} (t, z, z')|^2$, where $\bullet = \low$ or $k \in \{1, \dots, N\}$. Then, for any $r \in [1, 2)$, the integral operators
\begin{align*}
T_{\bullet}: f \mapsto \int_{M^{\circ}} J_{\bullet} (t, z, z') f(z') dg(z')
\end{align*}
satisfy $\|T_{\bullet} f\|_{r'} \lesssim |t|^{-3} \|f\|_r$ uniformly in $|t| \ge 1$.
\end{lemma}
\begin{proof}
First notice that we have $|J_{\bullet} (t, z, z')| \lesssim |t|^{-3} (1+d(z, z'))^{d-1} e^{-(d-1)d(z, z')}$. Let $U$ and $\tilde{U}$ be small neighborhoods of the front face FF of $M^2 _0$ satisfying $\tilde{U} \subset U$ and $\chi$ be a smooth cutoff function of FF such that $\chi =1$ in $\tilde{U}$ and supported in $U$. Since $d(z, z') + \log (xx')$ is uniformly bounded away from $\tilde{U}$ (see Proposition 3.4 in \cite{CheHa}) we have
\begin{align*}
\left\| \int_{M^{\circ}} J_{\bullet} (t, z, z')(1-\chi (z, z')) f(z')dg(z') \right\|_{r'} \lesssim \|J_{\bullet} (t, \cdot, \cdot)(1-\chi (\cdot, \cdot))\|_{L^{r'} (M^2 _0 \setminus \tilde{U})} \|f\|_r
\end{align*}
and the first factor can be estimated as
\begin{align*}
\|J_{\bullet} (t, \cdot, \cdot)(1-\chi (\cdot, \cdot))\|_{L^{r'} (M^2 _0 \setminus \tilde{U})} &\lesssim |t|^{-3} \left( \int (1+d(z, z'))^{(d-1)r'} e^{-(d-1)d(z, z')r'} dg(z)dg(z') \right)^{\frac{1}{r'}} \\
& \lesssim |t|^{-3} \left( \int_{0}^{1} \int_{0}^{1} \langle \log (xx')\rangle^{(d-1)r'} (xx')^{(d-1)r'} \frac{dx}{x^{d-1}} \frac{dx'}{{x'}^{d-1}} \right)^{\frac{1}{r'}} \\
& \lesssim |t|^{-3}.
\end{align*}
On the other hand we decompose $U$ into finitely many $U_i \subset U$ on which $x, x' \le \eta$ and $d_Y (y, y_i), d_Y (y', y_i) \le \eta$ hold for some $y_i \in Y$. We use a local coordinate $(x, y)$ near $(0, y_i)$ to define $\phi _i : V_i := \{ x \le \eta, d_Y (y, y_i) \le \eta \} \to V'_i$, where $V' _i$ is a neighborhood of $(0, 0) \in \mathbb{H}^{d}$. Then $\phi _i$ induces a diffeomorphism $\Phi _i : U_i \to U'_i \subset (\mathbb{H}^d)^2 _0$. Now 
\begin{align*}
|\phi _i \circ J_{\bullet} (t, \cdot, \cdot) \chi (\cdot, \cdot) \chi_{U_i} (\cdot, \cdot) \circ \phi^{-1} _i| \lesssim |t|^{-3} (1+r)^{d-1} e^{-(d-1)r}
\end{align*}
holds with the geodesic distance $r$ on $\mathbb{H}^d$. This integral kernel induces a bounded operator $L^r (V'_i) \to L^{r'} (V'_i)$ by Lemma 4.1 in \cite{AP2}. Since the pullback by $\phi _i$ and $\phi^{-1} _i$ are bounded operators between $L^r (V_i)$ and $L^r (V' _i)$, $J_{\bullet} (t, \cdot, \cdot) \chi (\cdot, \cdot) \chi_{U_i} (\cdot, \cdot)$ induces a bounded operator: $L^r (V_i) \to L^{r'} (V_i)$. By summing up in $i$ we obtain
\begin{align*}
\left\| \int_{M^{\circ}} J_{\bullet} (t, z, z') \chi (z, z') f(z')dg(z') \right\|_{r'} \lesssim |t|^{-3} \|f\|_r
\end{align*}
for any $f \in L^r (M^{\circ})$.
\end{proof}
As in Subsection \ref{25271431}, we set
\begin{align}
T^{\bullet} _{\omega, \epsilon} F (t) = T^{\bullet, 1} _{\omega, \epsilon} F (t)+ T^{\bullet, 2} _{\omega, \epsilon} F (t) &= \int_{\R} \chi_{\{ \epsilon < |t-s|<1\}} (t-s)^{\omega} U_{\bullet} (t) U_{\bullet} (s)^* F(s)ds \notag \\
& \quad \quad +\int_{\R} \chi_{\{ 1 < |t-s|<\frac{1}{\epsilon}\}} (t-s)^{\omega} U_{\bullet} (t) U_{\bullet} (s)^* F(s)ds, \label{25271421}
\end{align}
$T^{\bullet} _{\epsilon} = T^{\bullet} _{0, \epsilon}$ and $T^{\bullet} F(t) = \int_{\R} U_{\bullet} (t) U_{\bullet} (s)^* F(s)ds$ for spacetime simple functions $F$. Recall that, for $a \in [2, \infty]$, $\tilde{a} \in [2, \infty]$ is defined by $\tilde{a} = 2(\frac{a}{2})'$. First we consider the case $|t-s|>1$.
\begin{lemma}\label{25271434}
Let $(q, r) := \mathbb{P} (\frac{2(d+1)}{d}, \frac{2(d+1)}{d-1})$ (note that this is a $\frac{d}{2}$-admissible pair). Then
\begin{align}
\|W_1T^{\bullet, 2} _{0, \epsilon} W_2 \|_{\mathfrak{S}^{\beta'}} \lesssim \|W_1\|_{L^{\tilde{q}} _t L^{\tilde{r}} _z} \|W_2\|_{L^{\tilde{q}} _t L^{\tilde{r}} _z} \label{252101605}
\end{align}
holds for $\beta = \frac{2r}{r+2}$ uniformly in $\epsilon \in (0, 1)$.
\end{lemma}
\begin{proof}
Since the integral kernel of $U_{\bullet} (t) U_{\bullet} (s)^*$ is $K_{\bullet} (t-s, z, z')$ we have
\begin{align*}
\|W_1T^{\bullet, 2} _{0, \epsilon} W_2 \|^2 _{\mathfrak{S}^{2}} &= \int_{1 < |t-s|<\frac{1}{\epsilon}} |W_1 (t, z)|^2 J_{\bullet} (t-s, z, z') |W_2 (s, z')|^2 dsdtdg(z)dg(z') \\
&\lesssim \int_{|t-s|>1} \|W_1(t)\|^2 _{2r} \left( \int \left\| \int J_{\bullet} (t-s, z, z') |W_2 (s, z')|^2 dg(z') \right\|_{r'} ds \right) dt \\
& \lesssim  \int_{|t-s|>1} \|W_1(t)\|^2 _{2r} |t-s|^{-3} \|W_2 (s)\|^2 _{2r} dsdt \\
& \lesssim \|W_1\|^2 _{L^{2q} _t L^{2r} _z} \|(\chi_{\{ |\cdot| \ge 1\}} |\cdot|^{-3}) * \|W_2\|^2 _{2r}\|_{q'} \lesssim \|W_1\|^2 _{L^{2q} _t L^{2r} _z} \|W_2\|^2 _{L^{2q} _t L^{2r} _z}
\end{align*}
for $r \in [1, 2)$ and $q \in [1, 2]$ by Lemma \ref{25251030}. Hence for $Q \in [4, \infty]$ and $R \in (4, \infty]$,
\begin{align}
\|W_1T^{\bullet, 2} _{0, \epsilon} W_2 \|_{\mathfrak{S}^{2}} \lesssim \|W_1\|_{L^{\tilde{Q}} _t L^{\tilde{R}} _z} \|W_2\|_{L^{\tilde{Q}} _t L^{\tilde{R}} _z} \label{25271545}
\end{align}
holds. Next we take $(q_0, r_0)$ such that $q_0 \ge 2$, $r_0 >2$ and $\frac{2}{q_0} \ge d(\frac{1}{2} - \frac{1}{r_0})$. Then we have $\|W_1T^{\bullet, 2} _{0, \epsilon} W_2 \|_{\mathfrak{S}^{\infty}} \lesssim \|W_1\|_{L^{\tilde{q_0}} _t L^{\tilde{r_0}} _z} \|W_2\|_{L^{\tilde{q_0}} _t L^{\tilde{r_0}} _z} \|T^{\bullet, 2} _{0, \epsilon}\|_{L^{q'_0} _t L^{r'_0} _z \to L^{q_0} _t L^{r_0} _z}$ by H\"older's inequality. By the definition of $T^{\bullet, 2} _{0, \epsilon}$ we obtain
\begin{align*}
\|T^{\bullet, 2} _{0, \epsilon} F\|_{L^{q_0} _t L^{r_0} _z} &= \left\| \int_{\{1 < |t-s|<\frac{1}{\epsilon}\}} U_{\bullet} (t) U_{\bullet} (s)^* F(s)ds \right\|_{L^{q_0} _t L^{r_0} _z} \\
& \lesssim \left( \int \left( \int_{\{1 < |t-s|<\frac{1}{\epsilon}\}} \|U_{\bullet} (t) U_{\bullet} (s)^* F(s)\|_{r_0} ds \right)^{q_0} dt \right)^{\frac{1}{q_0}} \\
& \lesssim \|(\chi_{\{ |\cdot| \ge 1\}} |\cdot|^{-3}) * \|F(\cdot)\|_{r'_0} \|_{q_0} \lesssim \|F\|_{L^{q'_0} _t L^{r'_0} _z}
\end{align*}
uniformly in $\epsilon \in (0, 1)$, which yields
\begin{align}
\|W_1T^{\bullet, 2} _{0, \epsilon} W_2 \|_{\mathfrak{S}^{\infty}} \lesssim \|W_1\|_{L^{\tilde{q_0}} _t L^{\tilde{r_0}} _z} \|W_2\|_{L^{\tilde{q_0}} _t L^{\tilde{r_0}} _z}. \label{2528938}
\end{align}
Now in the $(\frac{1}{q}, \frac{1}{r})$ plane we take $\mathbb{A}(\frac{1}{4}, \frac{1}{4} - \eta)$ for sufficiently small $\eta >0$ such that the line $\mathbb{P} \mathbb{A}$ intersects with the line $\{\frac{1}{q} = \frac{1}{2}\}$ at $\mathbb{B}$ which satisfies the same conditions as $(q_0, r_0)$. Then we can use the complex interpolation for (\ref{25271545}) at $\mathbb{A}$ and (\ref{2528938}) at $\mathbb{B}$ to obtain the desired estimate.
\end{proof}
In the following situation, $|t-s|<1$, the proof is identical to that in asymptotically conic manifolds. Hence we omit details here.
\begin{lemma}\label{2528955}
Let $(q, r) \in [2, \infty]^2$ satisfy $\frac{2}{q} = d(\frac{1}{2} - \frac{1}{r})$, $r \in [2, \frac{2(d+1)}{d-1})$ and $\beta = \frac{2r}{r+2}$. Then
\begin{align*}
\|W_1T^{\bullet, 1} _{0, \epsilon} W_2 \|_{\mathfrak{S}^{\beta'}} \lesssim \|W_1\|_{L^{\tilde{q}} _t L^{\tilde{r}} _z} \|W_2\|_{L^{\tilde{q}} _t L^{\tilde{r}} _z}
\end{align*} 
holds uniformly in $\epsilon \in (0, 1)$.
\end{lemma}
\begin{proof}
Since $\|U_{\bullet} (t) U_{\bullet} (s)^*\|_{1 \to \infty} \lesssim |t-s|^{-\frac{d}{2}}$ holds for $|t-s| <1$, we have
\begin{align*}
\|W_1T^{\bullet, 1} _{0, \epsilon} W_2 \|_{\mathfrak{S}^{\beta'}} \lesssim \|W_1\|_{L^{\tilde{q}, 2\tilde{r}} _t L^{\tilde{r}} _z} \|W_2\|_{L^{\tilde{q}, 2\tilde{r}} _t L^{\tilde{r}} _z} \lesssim \|W_1\|_{L^{\tilde{q}} _t L^{\tilde{r}} _z} \|W_2\|_{L^{\tilde{q}} _t L^{\tilde{r}} _z}
\end{align*}
if $\frac{2(d+2)}{d} <r< \frac{2(d+1)}{d-1}$ by an almost same argument as in the proof of Proposition \ref{251241557} with $T^{\bullet, 1} _{\omega, \epsilon}$. At $(q, r) = (\infty, 2)$,
\begin{align*}
\|W_1T^{\bullet, 1} _{0, \epsilon} W_2 \|_{\mathfrak{S}^{\infty}} \lesssim \|W_1\|_{L^{2} _t L^{\infty} _z} \|W_2\|_{L^{2} _t L^{\infty} _z} \|T^{\bullet, 1} _{0, \epsilon}\|_{L^1 _t L^2 _z \to L^{\infty} _t L^2 _z} \lesssim \|W_1\|_{L^{2} _t L^{\infty} _z} \|W_2\|_{L^{2} _t L^{\infty} _z}
\end{align*}
holds since we know
\begin{align*}
\|T^{\bullet, 1} _{0, \epsilon} F\|_{L^{\infty} _t L^2 _z} = \left\| \int_{\{\epsilon <|t-s| < 1\}} U_{\bullet} (t) U_{\bullet} (s)^* F(s) ds \right\|_{L^{\infty} _t L^2 _z} \lesssim \|F\|_{L^1 _t L^2 _z}.
\end{align*}
Then by the complex interpolation we have the desired estimates.
\end{proof}
Now by the complex interpolation between (\ref{252101605}) and $\|W_1T^{\bullet, 2} _{0, \epsilon} W_2 \|_{\mathfrak{S}^{\infty}} \lesssim \|W_1\|_{L^{2} _t L^{\infty} _z} \|W_2\|_{L^{2} _t L^{\infty} _z}$ we obtain (\ref{252101605}) for any $(q, r) \in [2, \infty]^2$ and $\beta \in [1, \infty]$ satisfying $\frac{2}{q} = d(\frac{1}{2} - \frac{1}{r})$, $r \in [2, \frac{2(d+1)}{d-1})$ and $\beta = \frac{2r}{r+2}$. Combining these estimates with Lemma \ref{2528955} we have
\begin{align*}
\|W_1T^{\bullet} _{\epsilon} W_2 \|_{\mathfrak{S}^{\beta'}} \lesssim \|W_1\|_{L^{\tilde{q}} _t L^{\tilde{r}} _z} \|W_2\|_{L^{\tilde{q}} _t L^{\tilde{r}} _z}
\end{align*}
for the same exponents $(q, r, \beta)$. Then by a limiting argument as in the proof of Proposition \ref{251241557} we obtain
\begin{align}
\|W_1T^{\bullet} W_2 \|_{\mathfrak{S}^{\beta'}} \lesssim \|W_1\|_{L^{\tilde{q}} _t L^{\tilde{r}} _z} \|W_2\|_{L^{\tilde{q}} _t L^{\tilde{r}} _z} \label{252101623}
\end{align}
for the same exponents $(q, r, \beta)$.
\begin{proof}[Proof of Theorem \ref{252101626}]
In the case (ii), (\ref{252101709}) follows from summing up (\ref{252101623}) with respect to $\bullet$. If $\frac{2}{q} = d(\frac{1}{2} - \frac{1}{r})$ and $\frac{1}{r} \le \frac{2}{q} - \frac{1}{2}$, the complex interpolation between (\ref{252101623}) and $\|W_1T^{\bullet} W_2\|_{\mathfrak{S}^{\infty}} \lesssim \|W_1\|_{L^{\widetilde{q_{end}}} _t L^{\widetilde{r_{end}}} _z} \|W_2\|_{L^{\widetilde{q_{end}}} _t L^{\widetilde{r_{end}}} _z}$ yields the desired estimates in (iii), where $(q_{end}, r_{end}) = (2, \frac{2d}{d-2})$. Next by interpolating the dual of (\ref{252101709}) at $(\frac{1}{q}, \frac{1}{r}) = \mathbb{P} (\frac{d}{2(d+1)}, \frac{d-1}{2(d+1)})$, which is already proved above, and $\|W_1T^{\bullet} W_2\|_{\mathfrak{S}^{\infty}} \lesssim \|W_1\|_{L^{\infty} _t L^{\tilde{r}} _z} \|W_2\|_{L^{\infty} _t L^{\tilde{r}} _z}$ at $(\frac{1}{q}, \frac{1}{r}) = (\frac{1}{2}, \frac{1}{r})$, $r \in (2, \frac{2d}{d-2}]$ we have the desired estimates for $(q, r)$ in (iii) satisfying $\frac{1}{r} < \frac{2}{q} - \frac{1}{2}$. On the line $\frac{1}{r} = \frac{2}{q} - \frac{1}{2}$, we interpolate the dual of (\ref{252101709}) at $\mathbb{P}$ and $\|W_1T^{\bullet} W_2 \|_{\mathfrak{S}^{\infty}} \lesssim \|W_1\|_{L^{\tilde{q}} _t L^{\tilde{r}} _z} \|W_2\|_{L^{\tilde{q}} _t L^{\tilde{r}} _z}$ at $(\frac{1}{q}, \frac{1}{r})$ on $\frac{1}{r} = \frac{2}{q} - \frac{1}{2}$, which is arbitrarily close to $(\frac{1}{2}, \frac{1}{2})$, to obtain the desired estimates in (iii). This completes the proof of the case (iii). Finally we consider the case (i). For any $(q_1, r_1)$ in (i), we take points $C$ and $D$ in the $(\frac{1}{q}, \frac{1}{r})$ plane such that $(\frac{1}{q_1}, \frac{1}{r_1})$, $C$ and $D$ are on a line $\ell$ parallel to $\frac{1}{r} = \frac{2}{q} - \frac{1}{2}$, $C$ is on $\frac{2}{q} = d(\frac{1}{2} - \frac{1}{r})$ and $D$ is on $\frac{1}{r} = \frac{1}{2}$. Then by interpolating the dual of (\ref{252101709}) at $C$ (the case (ii)) and $\|W_1T^{\bullet} W_2 \|_{\mathfrak{S}^{\infty}} \lesssim \|W_1\|_{L^{\tilde{q_2}} _t L^{\tilde{r_2}} _z} \|W_2\|_{L^{\tilde{q_2}} _t L^{\tilde{r_2}} _z}$ at $E(\frac{1}{q_2}, \frac{1}{r_2})$ which lies on $\ell$ and sufficiently close to $D$, we finish the proof of (i).
\end{proof}

\section{Semiclassical limit of orthonormal Strichartz estimates}\label{2412182324}
In this section we prove a quantum and classical correspondence for the Strichartz estimates. Since we do not use a boundary decomposition, we use coordinates $(x, \xi) \in T^* \R^d$ as usual. The reason we use the scattering calculus is that estimates on the compactness of operators are important. Before giving a proof, we consider a rough explanation why the orthonormal Strichartz estimates (smoothing estimates on the Heisenberg equation restricted to the diagonal) yield the Strichartz estimates for the transport equation (without any restriction). We focus on (\ref{253171257}) and (\ref{253171256}) in {\S}\ref{25371356}. Suppose the initial data is $\gamma _{0} = a^w (x, hD_x)$ with $a \in S^{-0, -0}$ (any operator in $\mathfrak{S}^{\beta}$ can be approximated by such $\Psi$DOs in $\mathfrak{S}^{\beta}$ so we assume that from the beggining). Then, by an Egorov-type theorem, $\gamma (t) \sim b^w _{t} (x, hD_x) =: B_t$ with $b_{t} (x, \xi) = a \circ e^{-tH_p} (x, \xi)$ and $p(x, \xi) = g^{ij} (x) \xi_{i} \xi_{j}$. Since the Strichartz estimates capture smoothing effects in $L^p$ space, the relevant part of $\gamma (t, x, x')$ is the most singular part. However we know
\begin{align*}
\WF (K_{B_t}) \subset N^* \diag \setminus 0 = \{(x, \xi, x, -\xi) \mid \xi \neq 0\},
\end{align*}
where $K_{B_t}$ is the Schwartz kernel of $B_t$. Hence the Strichartz estimates for $\gamma (t, x, x)$ (restricted part to the diagonal) should have same amount of information as those for $\gamma (t, x, x')$ (without restriction). Therefore the orthonormal Strichartz estimates (estimates for $\gamma (t, x, x)$) are sufficient to deduce the Strichartz estimates for the transport equation (estimates for $b_t (x, \xi)$). Now we give a rigorous proof. First we construct a parametrix.
\begin{lemma}\label{2412231610}
Let $f \in C^{\infty} _{0} (T^* \R^d)$ and $T>0$. Then, for any $N \in \N$, there exist $F_N (t) = \sum_{j=0}^{N} \psi_{j} (t) \in S^{0, 0}$ and $\psi_{j} (t) \in h^j S^{0, 0}$ such that $F_N (0) = f^w (x, hD_x)$ and
\begin{align*}
\frac{\partial}{\partial t} F^w _N (t, x, hD_x) + \frac{i}{h} [p^w (x, hD_x), F^w _N (t, x, hD_x)] \in h^N \Op S^{-N, -N}
\end{align*}
uniformly in $t \in [-T, T]$.
\end{lemma}

\begin{proof}
The first approximation is given by
\begin{align*}
\frac{\partial}{\partial t} \psi _0 (t, x, \xi) + H_p \psi_0 (t, x, \xi) =0, \quad \psi_0 (0) =f.
\end{align*}
Hence $\psi_0$ is given by $\psi_0 (t) = f \circ e^{-tH_p} \in S^{0, 0}$. We note that, since $|t| \le T$, actually $\psi_0 (t)$ is compactly supported in $T^* \R^d$. Then there exist $r_0 (t) \in hS^{0, 0}$ and $r' _0 (t) \in h^N S^{-N, -N}$ such that
\begin{align*}
\frac{\partial}{\partial t} \psi^{w} _0  (t, x, hD_x) + \frac{i}{h} [p^w (x, hD_x), \psi^{w} _0  (t, x, hD_x)] = r^{w} _0 (t, x, hD_x) + r'^{w} _0 (t, x, hD_x)
\end{align*}
and $\supp r_0 (t) \subset \supp \psi_0 (t)$. Next we consider the inhomogeneous transport equation:
\begin{align*}
\frac{\partial}{\partial t} \psi _1 (t, x, \xi) + H_p \psi_1 (t, x, \xi) =-r_0 (t, x, \xi), \quad \psi_1 (0) =0.
\end{align*}
The unique solution is given by
\begin{align*}
\psi_1 (t, x, \xi) = -\int_{0}^{t} r_0 (s) \circ e^{-(t-s)H_p} (x, \xi) ds
\end{align*}
and this satisfies $\supp \psi_1 (t) \subset \supp \psi_0 (t)$ by a support property of $r_0 (t)$. In particular $\psi_1 (t) \in h S^{0, 0}$ holds. Hence we have
\begin{align*}
\frac{\partial}{\partial t} \psi^{w} _1  (t, x, hD_x) + \frac{i}{h} [p^w (x, hD_x), \psi^{w} _1  (t, x, hD_x)] = -r^{w} _0 (t, x, hD_x) + r^{w} _1 (t, x, hD_x)+ r'^{w} _1 (t, x, hD_x)
\end{align*}
for some $r_1 (t) \in h^2 S^{0, 0}$ with $\supp r_1 (t) \subset \supp \psi_0 (t)$ and $r' _1 (t) \in h^N S^{-N, -N}$. By iterating this procedure, we obtain $\psi_j (t) \in h^j S^{0, 0}$ such that $\supp \psi_{j} (t) \subset \supp \psi_0 (t)$ and $F_N (t) = \sum_{j=0}^{N} \psi_{j} (t) \in S^{0, 0}$ satisfying
\begin{align*}
\frac{\partial}{\partial t} F^w _N (t, x, hD_x) + \frac{i}{h} [p^w (x, hD_x), F^w _N (t, x, hD_x)] &= \sum_{j=0}^{N} r'^w _j (t, x, hD_x) + r_N (t, x, hD_x) \\
& \in h^N \Op S^{-N, -N}.
\end{align*}
\end{proof}
\begin{remark}\label{253291424}
Usually Egorov-type theorems are effective up to the Ehrenfest time. However they are sufficient to prove a pointwise convergence as in the following proof.
\end{remark}
\begin{proof}[Proof of Theorem \ref{2412231555}]
By a density argument, we assume $f \in C^{\infty} _{0} (T^* \R^d)$.

Step 1. Let $|t| \le T$ for some fixed $T>0$. Then by Lemma \ref{2412231610} we have
\begin{align*}
&\frac{d}{dt} \{e^{i \frac{t}{h} p^w (x, hD_x)} F^w _N (t, x, hD_x) e^{-i \frac{t}{h} p^w (x, hD_x)} \} \\
&= e^{i \frac{t}{h} p^w (x, hD_x)} \left\{ \frac{\partial}{\partial t} F^w _N (t, x, hD_x) + \frac{i}{h} [p^w (x, hD_x), F^w _N (t, x, hD_x)] \right\} e^{-i \frac{t}{h} p^w (x, hD_x)} \\
& =: e^{i \frac{t}{h} p^w (x, hD_x)} R_N (t) e^{-i \frac{t}{h} p^w (x, hD_x)}.
\end{align*}
Since $F^w _N (0) (x, hD_x) = f^w (x, hD_x)$, we have
\begin{align*}
e^{-i \frac{t}{h} p^w (x, hD_x)} f^w (x, hD_x) e^{i \frac{t}{h} p^w (x, hD_x)} &= F^w _N (t, x, hD_x) - \int_{0}^{t} e^{-i \frac{t-s}{h} p^w (x, hD_x)} R_N e^{i \frac{t-s}{h} p^w (x, hD_x)} ds \\
& =: F^w _N (t, x, hD_x) + \frak{R}_N (t)
\end{align*}
if $|t| \le T$. First we consider $\rho [F^w _N (t, x, hD_x)] = \rho [\psi^w _0 (t, x, hD_x)] + \rho [F^w _N (t, x, hD_x) - \psi^w _0 (t, x, hD_x)]$. For the first term we have
\begin{align*}
\rho [\psi^w _0 (t, x, hD_x)] (x) = \frac{1}{(2\pi h)^d} \int_{\R^d} f \circ e^{-tH_p} (x, \xi) d\xi 
\end{align*}
since $f \circ e^{-tH_p} \in C^{\infty} _0 (T^* \R^d)$. For the second term, 
\begin{align*}
\rho [F^w _N (t, x, hD_x) - \psi^w _0 (t, x, hD_x)] (x) = \frac{1}{(2\pi h)^d} \int_{\R^d} \sum_{j=1}^{N} \psi_j (t, x, \xi) d\xi,
\end{align*}
$\supp \psi_{j} (t) \subset \supp \psi _0 (t)$ and $\psi_{j} (t) \in h^j S^{0, 0}$ yields
\begin{align*}
(2\pi h)^d \rho [F^w _N (t, x, hD_x) - \psi^w _0 (t, x, hD_x)] (x) \to 0
\end{align*}
for $a. e. x \in \R^d$ as $h \to 0$. Hence 
\begin{align*}
(2\pi h)^d \rho [F^w _N (t, x, hD_x)] (x) \to  \int_{\R^d} f \circ e^{-tH_p} (x, \xi) d\xi 
\end{align*}
for $a. e. x \in \R^d$. Now we estimate the remainder term. Let $g \in L^{\infty} _{c} (\{ |x| \le R \})$ for large $R>0$. Then
\begin{align*}
\left| \int_{\R^d} \rho [\frak{R}_N (t)] (x) g(x) dx \right| &= \left| \tr \left[\int_{0}^{t} e^{-i \frac{t-s}{h} p^w (x, hD_x)} R_N (s) e^{i \frac{t-s}{h} p^w (x, hD_x)} ds g \right]\right| \\
& = \left| \int_{0}^{t} \tr \left[ R_N (s)  e^{i \frac{t-s}{h} p^w (x, hD_x)} g e^{-i \frac{t-s}{h} p^w (x, hD_x)}  \right] ds \right| \\
& \lesssim \int_{0}^{t} \|R_N (s)\|_{\mathfrak{S}^1} ds \|g\|_{L^{\infty} (\{ |x| \le R \})} \\
& = \mathcal{O} (h^{N-d}) \|g\|_{L^{\infty} (\{ |x| \le R \})}
\end{align*}
holds if $N$ is sufficiently large. Hence we have
\begin{align*}
\|  \rho [\frak{R}_N (t)] \|_{L^1 (\{ |x| \le R \})} \lesssim \mathcal{O} (h^{N-d}).
\end{align*}
By taking a subsequence, for any $t \in [-T, T]$, $(2\pi h_n)^d \rho [\frak{R}_N (t)]_{h=h_n} (x) \to 0$ for $a.e. x \in \{ |x| \le R \}$. Then, by using $\|  \rho [\frak{R}_N (t)]_{h=h_n} \|_{L^1 (\{ |x| \le 2R \})} \lesssim \mathcal{O} (h^{N-d} _n)$ again, we have a subsequence of $\{h_n\}$ (which we also denote by $\{h_n\}$) such that for any $t \in [-T, T]$, $(2\pi h_n)^d \rho [\frak{R}_N (t)]_{h=h_n} (x) \to 0$ for $a.e. x \in \{ |x| \le 2R \}$. By iterating this procedure and using a diagonal argument, we have a subsequence $\{h_n\}$ such that for any $t \in [-T, T]$, $(2\pi h_n)^d \rho [\frak{R}_N (t)]_{h=h_n} (x) \to 0$ for $a.e. x \in \R^d$. Therefore
\begin{align*}
(2\pi h_n)^d \rho [e^{-i \frac{t}{h_n} p^w (x, h_n D_x)} f^w (x, h_n D_x) e^{i \frac{t}{h_n} p^w (x, h_n D_x)}] (x) \to  \int_{\R^d} f \circ e^{-tH_p} (x, \xi) d\xi  
\end{align*}
for $a.e. (t, x) \in [-T, T] \times \R^d$. By iterating this procedure using $2T$ instead of $T$ and starting with $h_n$ instead of $h$, we have a subsequence of $\{h_n \}$ (again which we denote by $\{h_n\}$) such that
\begin{align*}
(2\pi h_n)^d \rho [e^{-i \frac{t}{h_n} p^w (x, h_n D_x)} f^w (x, h_n D_x) e^{i \frac{t}{h_n} p^w (x, h_n D_x)}] (x) \to  \int_{\R^d} f \circ e^{-tH_p} (x, \xi) d\xi  
\end{align*}
for $a.e. (t, x) \in [-2T, 2T] \times \R^d$. Then, by repeating this construction and using a diagonal argument as before, we have a subsequence $\{h_n\}$ such that
\begin{align}
(2\pi h_n)^d \rho [e^{-i \frac{t}{h_n} p^w (x, h_n D_x)} f^w (x, h_n D_x) e^{i \frac{t}{h_n} p^w (x, h_n D_x)}] (x) \to  \int_{\R^d} f \circ e^{-tH_p} (x, \xi) d\xi \label{2412232203}
\end{align}
for $a.e. (t, x) \in \R \times \R^d$.

Step 2. By the orthonormal Strichartz estimates and the homogeneity of $p$, we have
\begin{align*}
&\| (2\pi h_n)^d \rho [e^{-i \frac{t}{h_n} p^w (x, h_n D_x)} f^w (x, h_n D_x) e^{i \frac{t}{h_n} p^w (x, h_n D_x)}] \|_{L^{\frac{q}{2}} _t  L^{\frac{r}{2}} _x} \\
& \lesssim h^{d-\frac{2}{q}} _n \| \rho [e^{-i t p^w (x, D_x)} f^w (x, h_n D_x) e^{i t p^w (x, D_x)}] \|_{L^{\frac{q}{2}} _t  L^{\frac{r}{2}} _x} \\
& \lesssim h^{d-\frac{2}{q}} _n \| f^w (x, h_n D_x) \|_{\mathfrak{S}^{\beta}} \\
& \lesssim h^{d- \frac{2}{q} - \frac{d}{\beta}} _n \left(\|f\|_{L^{\beta} _{x, \xi}} + \mathcal{O} (h^{\frac{1}{2}} _n) \right) = \|f\|_{L^{\beta} _{x, \xi}} + \mathcal{O} (h^{\frac{1}{2}} _n)
\end{align*}
since $d- \frac{2}{q} - \frac{d}{\beta} =0$. Here we have used $\|a^w (x, hD_x)\|_{\mathfrak{S}^{\beta}} \lesssim \|a\|_{L^{\beta} _{x, \xi}} + \mathcal{O} (h^{\frac{1}{2}})$ (see \cite{BHLNS} (5.8)). If $q < \infty$, by (\ref{2412232203}) and Fatou's lemma, we obtain
\begin{align*}
&\left\| \int_{\R^d} f \circ e^{-tH_p} (x, \xi) d\xi \right\|_{L^{\frac{q}{2}} _t  L^{\frac{r}{2}} _x} \\
& \le \liminf _{n \to \infty} \| (2\pi h_n)^d \rho [e^{-i \frac{t}{h_n} p^w (x, h_n D_x)} f^w (x, h_n D_x) e^{i \frac{t}{h_n} p^w (x, h_n D_x)}] \|_{L^{\frac{q}{2}} _t  L^{\frac{r}{2}} _x} \\
& \lesssim  \|f\|_{L^{\beta} _{x, \xi}}.
\end{align*}
If $q = \infty$, by the Liouville theorem, 
\begin{align*}
\left\| \int_{\R^d} f \circ e^{-tH_p} (x, \xi) d\xi \right\|_{L^{\infty} _t L^1 _x} \le \sup _{t \in \R} \int_{T^* \R^d} |f| \circ e^{-tH_p} (x, \xi) dxd\xi = \|f\|_{L^1 _{x, \xi}}.
\end{align*}
This completes the proof of Theorem \ref{2412231555}.
\end{proof}
Now we show the homogeneous and inhomogeneous Strichartz estimates for the transport equation. Theorem \ref{253201908} follows from Corollary \ref{2412232213} and \ref{2412232345}.
\begin{corollary}\label{2412232213}
Let $p \in S^{2, 0}$ be as in Theorem \ref{2412231555}. Suppose (\ref{2412231605}) holds for any $(q, r, \beta)$ satisfying $q, r \in [2, \infty], r \in [2, \frac{2(d+1)}{d-1}), \frac{2}{q} = d(\frac{1}{2} - \frac{1}{r})$ and $\beta = \frac{2r}{r+2}$. Then
\begin{align}
\|f \circ e^{-tH_p}\|_{L^q _t L^r _x L^p _{\xi}} \lesssim \|f\|_{L^a _{x, \xi}}\label{2412232304}
\end{align}
holds for any non-endpoint KT-admissible quadruplet $(q, r, p, a)$.
\end{corollary}

\begin{proof}
Let $(q, r, p, a)$ be a non-endpoint KT-admissible quadruplet. If $a = \infty$, then $q=r=p=\infty$ by the KT-admissible condition. Hence (\ref{2412232304}) is obvious in this case. If $q = \infty$, then $p=r=a$ and (\ref{2412232304}) follows from the Liouville theorem. If $p = \infty$ or $r = \infty$, then $a= \infty$ by the KT-admissible condition and contained in the first case. In the other cases, all of $q, r, p, a$ are finite and (\ref{2412232304}) is equivalent to
\begin{align*}
\left\| \int_{\R^d} f \circ e^{-tH_p} (x, \xi) d\xi \right\|_{L^{\frac{q}{p}} _t  L^{\frac{r}{p}} _x} \lesssim \|f\|_{L^{\frac{a}{p}} _{x, \xi}}.
\end{align*}
This estimate follows from Theorem \ref{2412231555}.
\end{proof}

\begin{corollary}\label{2412232345}
We assume the same conditions as in Corollary \ref{2412232213}. If $(q, r, p, a)$ and $(\tilde{q}, \tilde{r}, \tilde{p}, a')$ are non-endpoint KT-admissible quadruplets, then
\begin{align}
\left\| \int_{0}^{t} F(s) \circ e^{-(t-s)H_p} (x, \xi) ds \right\|_{L^q _t L^r _x L^p _{\xi}} \lesssim \|F\|_{L^{\tilde{q}'} _t L^{\tilde{r}'} _x L^{\tilde{p}'} _{\xi}} \label{2412232355}
\end{align}
holds.
\end{corollary}

\begin{proof}
The argument is almost the same as in the case of the free Hamiltonian (see \cite{O1} ). For $T: L^{a'} _{x, \xi} \to {L^{\tilde{q}} _t L^{\tilde{r}} _x L^{\tilde{p}} _{\xi}}$, $ f \mapsto f \circ e^{-tH_p}$, its formal adjoint is given by
\begin{align*}
T^* F(x, \xi) = \int_{\R} F(s) \circ e^{sH_p} (x, \xi) ds,
\end{align*}
$T^* : {L^{\tilde{q}'} _t L^{\tilde{r}'} _x L^{\tilde{p}'} _{\xi}} \to L^{a} _{x, \xi}$, which yields
\begin{align*}
TT^* F(t) = \int_{\R} F(s) \circ e^{-(t-s)H_p} (x, \xi) ds.
\end{align*}
Then we obtain
\begin{align*}
\left\| \int_{0}^{t} F(s) \circ e^{-(t-s)H_p} (x, \xi) ds \right\|_{L^q _t L^r _x L^p _{\xi}} &\le \left\| \int_{\R} |F(s)| \circ e^{-(t-s)H_p} (x, \xi) ds \right\|_{L^q _t L^r _x L^p _{\xi}} \\
& = \|TT^* |F|\|_{L^q _t L^r _x L^p _{\xi}} \le \|T\| \|T^*\| \|F\|_{L^{\tilde{q}'} _t L^{\tilde{r}'} _x L^{\tilde{p}'} _{\xi}},
\end{align*}
where $\|T\| = \|T\|_{L^{a} _{x, \xi} \to L^q _t L^r _x L^p _{\xi}}$ and $\|T^*\| = \|T^*\|_{{L^{\tilde{q}'} _t L^{\tilde{r}'} _x L^{\tilde{p}'} _{\xi}} \to L^{a} _{x, \xi}}$.
\end{proof}
Combining Theorem \ref{253201908} with the orthonormal Strichartz estimates proved in the previous section we obtain the following result, which implies Corollary \ref{253202133}.
\begin{corollary}\label{251282306}
Let $(M^{\circ}, g) = (\R^d, g)$ be a nontrapping scattering manifold satisfying Assumption \ref{25121723}, where $|\partial^{\alpha} _z g_{ij} (z)| \lesssim \langle z \rangle^{-|\alpha|}$ holds for any $1 \le i, j \le d$ and $\alpha \in \N^d _0$. Then, for $p(z, \zeta) = \sum_{i, j=1}^{d} g^{ij} (z) \zeta _{i} \zeta _{j} \in S^{2, 0}$ and $(q, r, \beta)$ satisfying $q, r \in [2, \infty], r \in [2, \frac{2(d+1)}{d-1}), \frac{2}{q} = d(\frac{1}{2} - \frac{1}{r})$ and $\beta = \frac{2r}{r+2}$, we have
\begin{align*}
\left\| \int_{\R^d} f \circ e^{-tH_p} (z, \zeta) d\zeta \right\|_{L^{\frac{q}{2}} _t  L^{\frac{r}{2}} (\R^d, dg)} \lesssim \|f\|_{L^{\beta} (T^* \R^d, dgd\zeta)}.
\end{align*}
Therefore (\ref{2412232304}) and (\ref{2412232355}) hold for any non-endpoint KT-admissible quadruplet. 
\end{corollary}
\begin{proof}
First of all, notice that since $\deter g \sim 1$ uniformly on $\R^d$, the norms $\| \cdot \|_{L^p (\R^d, dg)} = \| \cdot \|_{L^p (\R^d, \sqrt{g} dz)}$ and $\| \cdot \|_{L^p (\R^d, dz)}$ are equivalent. Since $P= -\frac{1}{\sqrt{g}} \partial_{i} g^{ij} \sqrt{g} \partial_{j}$ is not exactly $p^{w} (z, D_z)$ (there appears a lower order term) the above argument does not directly apply to the present case. However modifications are easy. As in Lemma \ref{2412231610}, we can construct $F_N (t, z, \zeta)$ such that
\begin{align*}
\frac{\partial}{\partial t} F^w _N (t, z, hD_z) + \frac{i}{h} [h^2 P, F^w _N (t, z, hD_z)] \in h^N \Op S^{-N, -N}
\end{align*}
since the principle symbol of $P$ is $p$. We define the density $\rho (A)_g$ of an operator $A \in \mathcal{B} (L^2 (\R^d, dg))$ (if exists) by $\tr (A \chi)_{L^2 (\R^d, g)} = \int_{\R^d} \rho (A)_g (z) \chi (z) dg(z)$ for any simple function $\chi$. Then $\rho (e^{itP} \gamma_{0} e^{-itP})_g = \sum_{j=0}^{\infty} \nu_{j} |e^{-itP}  f_j|^2$ holds for $\gamma_{0} =\sum_{j=0}^{\infty} \nu_{j} | f_j \rangle \langle f_j |$, where $\{f_j \}$ is any orthonormal system in $L^2 (\R^d, dg)$. Combining this with Theorem \ref{251241323} we have 
\begin{align*}
\| \rho (e^{itP} \gamma_{0} e^{-itP})_g \|_{L^{\frac{q}{2}} _t  L^{\frac{r}{2}} (\R^d, dg)} \lesssim \| \gamma _{0} \|_{\mathfrak{S}^{\beta} (L^2 (\R^d, dg))}.
\end{align*}
Then other arguments are almost identical since we can estimate, for example,  
\begin{align*}
\| a^w (z, hD_z) \|_{\mathfrak{S}^{\beta} (L^2 (\R^d, dg))} \lesssim \| a^w (z, hD_z) \|_{\mathfrak{S}^{\beta} (L^2 (\R^d, dz^2))} 
&\lesssim \|a\|_{L^{\beta} (T^* \R^d, dzd\zeta)} + \mathcal{O} (h^{\frac{1}{2}}) \\
&\lesssim \|a\|_{L^{\beta} (T^* \R^d, dgd\zeta)} + \mathcal{O} (h^{\frac{1}{2}}).
\end{align*}
We omit details. 
\end{proof}
\begin{remark}\label{253812223}
As we have seen in this section, the quantum Strichartz estimates imply the classical Strichartz estimates. We consider the converse here. Assume $p(x, \xi) = |\xi|^2$. Then by the exact Egorov theorem, (\ref{253291637}) and (\ref{253291638}) are equivalent:
\begin{align}
\left\| \int_{\R^d} f (x-2t\xi, \xi) d\xi \right\|_{L^{\frac{q}{2}} _t  L^{\frac{r}{2}} _x} \lesssim \|f\|_X, \label{253291637} \\
\|\rho (e^{it\Delta} f^w (x, D_x) e^{-it\Delta})\|_{L^{\frac{q}{2}} _t  L^{\frac{r}{2}} _x} \lesssim \|f\|_X. \label{253291638}
\end{align}
Here $X$ is a function space, e.g. $L^p _{x, \xi}$ or the modulation space $M^{p, q}$ (see \cite{To} for the definition). Since $M^{\beta, \min\{\beta, \beta'\}} \subset s^w _{\beta} \subset M^{\beta, \max\{\beta, \beta'\}}$ holds (\cite[(4.19)]{To}), (\ref{253291637}) with $X=M^{\beta, \max\{\beta, \beta'\}}$ implies the orthonormal Strichartz estimates. Note that, by $M^{\beta, \min\{\beta, \beta'\}} \subset L^{\beta} _{x, \xi} \subset M^{\beta, \max\{\beta, \beta'\}}$, (\ref{253291637}) with $X=M^{\beta, \max\{\beta, \beta'\}}$ is stronger than the velocity average estimates ((\ref{253291637}) with $X= L^{\beta} _{x, \xi}$) and the author is not aware if it is derived from the orthonormal Strichartz estimates. For general symbol $p$, even in the modulation setting, the converse may be difficult since global-in-time estimates on the remainder term require additional arguments.
\end{remark}

\section{Negative results on scattering manifolds with trapped sets}\label{2412182325}
In this section we prove Theorem \ref{25371352}. A subset $K \subset T^* \R^d$ is called backward (forward) flow invariant if $e^{-tH_p} (K) \subset K$ holds for all $t \ge 0$ $(t \le 0)$. A canonical measure on $T^* \R^d$ obtained from the symplectic structure on $\R^d$ is denoted by $\mu$ (see \cite{DZ} Section 6). In this section we assume $p \in S^{2, 0}$ is real-valued, homogeneous of degree $2$, $H_p$ is complete on $T^* \R^d$ and $P = p^w (x, D_x)$ is essentially self-adjoint on $L^2 (\R^d)$ with its core $C^{\infty} _{0} (\R^d)$. Now we recall a definition of the stability of a periodic trajectory following \cite{RS}.
\begin{defn}{(\cite[Definition 3.8.1]{RS})}\label{253301237}
A periodic integral curve $\gamma$ of $H_p$ is called stable if the following condition is satisfied: For any neighborhood of $U$ of $\gamma$, there exists another neighborhood $V$ of $\gamma$ such that $e^{tH_p} (V) \subset U$ for all $t \in [0, \infty)$.
\end{defn}
See \cite[{\S}3.8 and 9.7]{RS} for sufficient conditions to ensure the stability, for example, the existence of a Lyapunov function, an assumption on the Floquet multipliers called \textit{elementary}. The latter condition is also assumed in \cite{Ra} to construct quasimodes.
\begin{proposition}\label{25171303}
If there exists a backward or forward flow invariant bounded subset $K \subset T^* \R^d$ satisfying $\mu (K) >0$, then
\begin{align*}
\|f \circ e^{-tH_p}\|_{L^q _t L^r _x L^p _{\xi}} \lesssim \|f\|_{L^a _{x, \xi}}
\end{align*}
fails for any $(q, r, p, a) \in [1, \infty) \times [1, \infty]^3$.
\end{proposition}
\begin{proof}
Let $\pi : T^* \R^d \to \R^d$ be the canonical projection. We take $f \in C^{\infty} _0 (T^* \R^d)$ such that $f \equiv 1$ near $K$. Since $K$ is bounded we have
\begin{align*}
\|f \circ e^{-tH_p}\|_{L^1 (K)} \lesssim \| \|f \circ e^{-tH_p}\|_{L^p _{\xi}} \|_{L^1 (\pi (K))} \lesssim \|f \circ e^{-tH_p}\|_{L^r _x L^p _{\xi}}.
\end{align*}
By definition of $f$, $\|f \circ e^{-tH_p}\|_{L^1 (K)} \ge \mu (K) >0$ holds for any $t \ge 0$ $(t \le 0)$. These estimates imply $\|f \circ e^{-tH_p}\|_{L^q _t L^r _x L^p _{\xi}} = \infty$.
\end{proof}
For reader's convenience we recall the definition of the failure of the orthonormal Strichartz estimates. In the following definition, ``sharp'' comes from the fact that for $P=-\Delta$, the largest (the best) $\beta$ satisfying (\ref{25171415}) is $\beta = \frac{2r}{r+2}$.
\begin{defn}\label{25171343}
We say that the sharp orthonormal Strichartz estimates fail if and only if for any $(q, r, \beta)$ satisfying $q, r \in [2, \infty), r \in [2, \frac{2(d+1)}{d-1}), \frac{2}{q} = d(\frac{1}{2} - \frac{1}{r})$ and $\beta = \frac{2r}{r+2}$,
\begin{align}
\left\| \sum_{j=0}^{\infty} \nu_j |e^{-itP}  f_j|^2 \right\|_{L^{\frac{q}{2}} _t  L^{\frac{r}{2}} _x} \lesssim \|\nu\|_{\ell^{\beta}}\label{25171415}
\end{align}
does not hold uniformly in orthonormal $\{f_j\} \subset L^2$ and $\nu = \{\nu_{j}\}$.
\end{defn}
In 1-dimensional case, just a periodic trajectory breaks the orthonormal Strichartz estimates. This is a special phenomenon coming from $\dimm T^* \R =2$.
\begin{corollary}\label{25171400}
Assume $d=1$. If there exists a periodic trajectory $\gamma \subset T^* \R$ associated to $H_p$, the sharp orthonormal Strichartz estimates fail.
\end{corollary}
\begin{proof}
By the Jordan curve theorem there exists a bounded connected region $\Omega \subset T^* \R$ such that $\partial \Omega = \gamma$. Then we can take $K := \Omega \cup \gamma$ as a forward and backward flow invariant set in Proposition \ref{25171303}. If we assume  (\ref{25171415}) holds for some $(q, r, \beta)$ as in Definition \ref{25171343}, then Theorem \ref{2412231555} yields $\|f \circ e^{-tH_p}\|_{L^{\frac{q}{2}} _t L^{\frac{r}{2}} _x L^1 _{\xi}} \lesssim \|f\|_{L^{\beta} _{x, \xi}}$. However this contradicts the conclusion in Proposition \ref{25171303}.
\end{proof}
Next we construct a Riemannian metric as in Theorem \ref{25371352} (iii).
We consider two diffeomorphisms $F$ and $G$. Let $N$ be the north pole of $\mathbb{S}^{d}$ for $d \ge 2$. Then we define $F: \mathbb{S}^d \setminus \{N\} (\subset \R^{d+1}) \to \R^d$ by $F(x_1, \ldots, x_{d+1}) = (\frac{x_1}{1-x_{d+1}}, \ldots, \frac{x_d}{1-x_{d+1}})$ and $G: \R^d \to \mathbb{B}^d$ by $G(x) = \frac{x}{\langle x \rangle}$. Let $g_{\mathbb{S}^d \setminus \{N\}} = \iota ^* g_{\mathbb{S}^d}$ be the pull-back of the standard metric $g_{\mathbb{S}^d}$ on $\mathbb{S}^d$ by the inclusion $\iota: \mathbb{S}^d \setminus \{N\} \hookrightarrow \mathbb{S}^d$. For $r_0$, $\epsilon \in (0, 1)$ satisfying $r_0 + 2\epsilon <1$, we take a cutoff function $\chi \in C^{\infty} ([0, \infty); [0, 1])$ such that 
\begin{align}
\chi (r) =
\begin{cases}
1 & (r \le r_0 + \epsilon) \\
0 & (r \ge r_0 +2\epsilon) \label{252131340}.
\end{cases}
\end{align}
\begin{proposition}\label{251151538}
If $r_0$ and $\epsilon$ are sufficiently small, then
\[g_{sc} := (1- \chi (|x|))dx^2 + \chi (|x|) (G^{-1})^* (F^{-1})^* g_{\mathbb{S}^d \setminus \{N\}}\]
is a well-defined scattering metric on $\R^d$. Furthermore the sharp orthonormal Strichartz estimates fail for $-\Delta_{g_{sc}}$.
\end{proposition}
\begin{proof}
Let $\pi : T^* (\mathbb{S}^d \setminus \{N\}) \to \mathbb{S}^d \setminus \{N\}$ be the canonical projection. For $X_0 = (1, 0, \ldots, 0) \in \mathbb{S}^d \setminus \{N\} (\subset \R^{d+1})$, we take a chart $W$ near $X_0$ such that $\pi ^{-1} (W) = W \times \R^d$. We set $\Xi_{0} = (0, 1, 0, \dots, 0) \in T^* _{X_0} W (\subset \R^{d+1})$. Then, for the geodesic flow $\Phi _{t} (X, \Xi)$ on $T^* (\mathbb{S}^d \setminus \{N\})$, $\pi (\Phi _{t} (X_0, \Xi_{0}))$ is a great circle. We show that there exists a neighborhood $U \subset W$ of $X_0$ and a neighborhood $V \subset \R^d $ of $\Xi_{0}$ such that $\pi (\{ \Phi_{t} (U\times V) \mid t \in \R\}) \subset \{x \in \R^{d+1} \mid |x_{d+1}| < \frac{1}{2}\}$. Since $\Phi _{t} (X, \Xi) = (y(t, X, \Xi), \eta (t, X, \Xi))$ satisfies $(y(t, X, \lambda \Xi), \eta (t, X, \lambda \Xi)) = (y(\lambda t, X, \Xi), \lambda \eta (\lambda t, X, \Xi ))$ for $\lambda >0$, the periods of $\Phi_{t} (X, \Xi)$, $(X, \Xi) \in U \times V$ are uniformly estimated from above and below by positive constants. Hence there exists $T >0$ satisfying $\{ \Phi_{t} (U\times V) \mid t \in \R\} = \{ \Phi_{t} (U\times V) \mid t \in [0, T]\}$. Suppose there exists $(X_n, \Xi_{n}) \in U \times V$ such that $(X_n, \Xi_{n}) \to (X_0, \Xi_{0})$ as $n \to \infty$ and $\pi (\Phi_{t_n} (X_n, \Xi_{n})) \notin \{ |x_{d+1}| < \frac{1}{2}\}$ for some $t_n \in [0, T]$. Then by passing to a subsequence we have $t_n \to t_{*}$ for some $t_{*} \in [0, T]$ and $\pi (\Phi_{t_n} (X_n, \Xi_{n})) \to \pi (\Phi_{t_{*}} (X_0, \Xi_{0})) \in \{x_{d+1} =0\}$. This contradicts the definition of $(X_n, \Xi_{n})$ and we have desired $U$ and $V$. Now we proceed as
\begin{align*}
&T^* (\mathbb{S}^d \setminus \{N\}) \supset U \times V \mapsto \sharp (U \times V) \subset T (\mathbb{S}^d \setminus \{N\}) \\
& \mapsto F \times dF (\sharp (U \times V)) \subset T\R^d \mapsto  (G \times dG) \circ (F \times dF) (\sharp (U \times V)) \subset T \mathbb{B}^d \\
& \mapsto \mathcal{X} := \flat ((G \times dG) \circ (F \times dF) (\sharp (U \times V))) \subset T^* \mathbb{B}^d.
\end{align*}
Here $\sharp$ and $\flat$ are musical isomorphisms on $T^* (\mathbb{S}^d \setminus \{N\})$ and $T \mathbb{B}^d$.
Since every map above is a diffeomorphism, we have $\VOL _{T^* \mathbb{B}^d} (\mathcal{X}) >0$. We set $K := \{\Psi_{t} (\mathcal{X}) \mid t \in \R\} \Subset T^* \mathbb{B}^d$, where $\Psi_{t}$ is the geodesic flow on $T^* \mathbb{B}^d$, and then $\VOL _{T^* \mathbb{B}^d} (K) >0$ holds. Since we know $\pi (\{ \Phi_{t} (U\times V) \mid t \in \R\}) \subset \{x \in \R^{d+1} \mid |x_{d+1}| < \frac{1}{2}\}$, there exists $r_0 \in (0, 1)$ such that $\pi (K) \subset \{|x| < r_0\}$. For this $r_0$, we take $\epsilon \in (0, 1)$ and $\chi$ as in (\ref{252131340}) to define $g_{sc}$. By a support property of $\chi$, $g_{sc}$ is a compactly supported perturbation of the Euclidean metric $dx^2$. Since the geodesic flow on $T^* \R^d$ associated with $g_{sc}$ is equal to $\Psi_{t}$ on $K \subset T^* \R^d$, we can apply Proposition \ref{25171303} and the rest of the argument is identical to that of Corollary \ref{25171400}.
\end{proof}
\begin{remark}\label{253301827}
The author believes this construction should work even if we replace the sphere with other manifolds with a stable geodesic since we do not use explicit formulas of $F$ and $G$.
\end{remark}
\begin{proof}[Proof of Theorem \ref{25371352}]
(i) and (iii) follow from Corollary \ref{25171400} and Proposition \ref{251151538}. (ii) follows from taking $K= \{e^{tH_p} (V) \mid t \ge 0\}$ in Proposition \ref{25171303}, where $V$ is from Definition \ref{253301237}.
\end{proof}
\section{Boltzmann equations on scattering manifolds}\label{2412182326}
In this section we prove Theorem \ref{25121757} following the case $p(x, \xi)=|\xi|^2$ in \cite{CDP, HJ1, HJ2}. For convenience we split the collisional term $Q$ into the gain term $Q^+$ and the loss term $Q^{-}$:
\begin{align}
&Q^{+} (f, g) (\xi) = \int_{\R^d} \int_{\mathbb{S}^{d-1}} f(\xi ')g(\xi ' _{*}) B(\xi - \xi_{*}, \omega) d \omega d \xi _{*} \label{25112305} \\
&Q^{-} (f, g) (\xi) = \int_{\R^d} \int_{\mathbb{S}^{d-1}} f(\xi) g(\xi _{*}) B(\xi - \xi_{*}, \omega) d \omega d \xi _{*} \label{24112306}.
\end{align}
Hence we have $Q (f, g) = Q^+ (f, g) - Q^{-} (f, g)$. First we consider the Boltzmann equation without loss term, which is easily handled by a contraction.
\begin{proposition}\label{25112308}
Assume $\gamma = 2-d$, $d =2, 3$, (\ref{2412232304}) and (\ref{2412232355}) for any non-endpoint KT-admissible quadruplets $(q, r, p, a)$ and $(\tilde{q}, \tilde{r}, \tilde{p}, a')$. Then the gain-only Boltzmann equation:
 \[
\left\{
\begin{array}{l}
\partial _{t} f(t, x, \xi) +H_p f(t, x, \xi) = Q^+ (f, f) (t, x, \xi) \\
f(0, x, \xi) = f_0 (x, \xi) \\
\end{array}
\right.
\tag{B+}\label{251123140}
\]
has a unique global solution $f \in C ([0, \infty); L^d _{x, \xi}) \cap L^{\boldsymbol{q}} ([0, \infty); L^{\boldsymbol{r}} _x L^{\boldsymbol{p}} _{\xi})$ for $(\boldsymbol{q}, \boldsymbol{r}, \boldsymbol{p}) \in \Lambda$ provided $\|f_0\|_{L^d _{x, \xi}}$ is sufficiently small. Furthermore there exists $f^{+} \in L^d _{x, \xi}$ such that
\begin{align*}
\|f(t) - f^{+} \circ e^{-tH_p} \|_{L^d _{x, \xi}} \to 0 \quad as \quad t \to \infty.
\end{align*}
\end{proposition}
\begin{proof}
For simplicity we write $L^{\boldsymbol{q}} _t L^{\boldsymbol{r}} _x L^{\boldsymbol{p}} _{\xi} = L^{\boldsymbol{q}} ([0, \infty); L^{\boldsymbol{r}} _x L^{\boldsymbol{p}} _{\xi})$.
For $(\boldsymbol{q}, \boldsymbol{r}, \boldsymbol{p}) \in \Lambda$, we define $(\tilde{\boldsymbol{q}}, \tilde{\boldsymbol{r}}, \tilde{\boldsymbol{p}})$ by $\tilde{\boldsymbol{r}} ' = \frac{{\boldsymbol{r}}}{2}$, $\tilde{\boldsymbol{q}} ' = \frac{{\boldsymbol{q}}}{2}$ and $\frac{1}{{\boldsymbol{p}}} + \frac{1}{{\boldsymbol{r}}} = \frac{1}{\tilde{\boldsymbol{p}} '} + \frac{1}{\tilde{\boldsymbol{r}} '}$. Since $(\boldsymbol{q}, \boldsymbol{r}, \boldsymbol{p}, d)$ is a KT-admissible quadruplet, $(\tilde{\boldsymbol{q}}, \tilde{\boldsymbol{r}}, \tilde{\boldsymbol{p}}, d')$ is also a KT-admissible quadruplet. We set
\begin{align*}
\Phi [f] = f_0 \circ e^{-tH_p} + \int_{0}^{t} Q^+ (f, f) \circ e^{-(t-s)H_p} ds
\end{align*}
and show that $\Phi$ is a contraction on $X_{\epsilon} = \{ f \in L^{\boldsymbol{q}} _t L^{\boldsymbol{r}} _x L^{\boldsymbol{p}} _{\xi} \mid \|f\|_{L^{\boldsymbol{q}} _t L^{\boldsymbol{r}} _x L^{\boldsymbol{p}} _{\xi} } < \epsilon \}$ for sufficiently small $\epsilon >0$.
Then by (\ref{2412232304}) and (\ref{2412232355}),
\begin{align*}
\| \Phi [f] \|_{L^{\boldsymbol{q}} _t L^{\boldsymbol{r}} _x L^{\boldsymbol{p}} _{\xi}} \lesssim \|f_0\|_{L^d _{x, \xi}} + \|Q^+ (f, f)\|_{L^{\tilde{\boldsymbol{q}} '} _t L^{\tilde{\boldsymbol{r}} '} _x L^{\tilde{\boldsymbol{p}} '} _{\xi}  }
\end{align*}
holds. Now we use the following inequality (Proposition 2.3 in {\cite{HJ1}}):
\begin{align*}
\|Q^{+} (f, g)\|_{L^{\tilde{\boldsymbol{p}} '} _{\xi}} \lesssim \|f\|_{L^{\boldsymbol{p}} _{\xi}} \|g\|_{L^{\boldsymbol{p}} _{\xi}}
\end{align*}
and obtain
\begin{align*}
\| \Phi [f] \|_{L^{\boldsymbol{q}} _t L^{\boldsymbol{r}} _x L^{\boldsymbol{p}} _{\xi}} \lesssim \|f_0\|_{L^d _{x, \xi}} + \| f \|^2 _{L^{\boldsymbol{q}} _t L^{\boldsymbol{r}} _x L^{\boldsymbol{p}} _{\xi}}.
\end{align*}
By a similar calculation, we also have an estimate for the difference:
\begin{align*}
\| \Phi [f] - \Phi [g] \|_{L^{\boldsymbol{q}} _t L^{\boldsymbol{r}} _x L^{\boldsymbol{p}} _{\xi}} \lesssim (\| f \| _{L^{\boldsymbol{q}} _t L^{\boldsymbol{r}} _x L^{\boldsymbol{p}} _{\xi}} + \| g\| _{L^{\boldsymbol{q}} _t L^{\boldsymbol{r}} _x L^{\boldsymbol{p}} _{\xi}}) \| f-g\|_{L^{\boldsymbol{q}} _t L^{\boldsymbol{r}} _x L^{\boldsymbol{p}} _{\xi}}.
\end{align*}
Hence if $\epsilon$ and $\|f_0\|_{L^d _{x, \xi}}$ are sufficiently small, $\Phi$ is a contraction on $X_{\epsilon}$. Next we show that the unique solution $f \in L^{\boldsymbol{q}} _t L^{\boldsymbol{r}} _x L^{\boldsymbol{p}} _{\xi}$ belongs to $C ([0, \infty); L^d _{x, \xi})$. Since $U(t)f_0 := f_0 \circ e^{-tH_p}$ satisfies $\| U(t)f_0 \|_{L^d _{x, \xi}} = \|f_0 \|_{L^d _{x, \xi}}$ (Liouville's theorem), $U(t)f_0 \in C ([0, \infty); L^d _{x, \xi})$ follows from the case $f_0 \in C^{\infty} _0 (T^* \R^d)$ (this is a consequence of the dominated convergence theorem). Moreover
\begin{align*}
\left| \left\langle \int_{\tau}^{t} U(-s) Q^{+} (f, f) (s)ds, \phi \right\rangle _{x, \xi} \right| &= \left| \int_{\tau}^{t} \left\langle Q^{+} (f, f) (s), U(s) \phi \right\rangle _{x, \xi} ds \right| \\
& \le \|U(t) \phi\|_{L^{\tilde{\boldsymbol{q}} } _t L^{\tilde{\boldsymbol{r}} } _x L^{\tilde{\boldsymbol{p}} } _{\xi}} \|Q^+ (f, f)\|_{L^{\tilde{\boldsymbol{q}} '} ([\tau, t]; L^{\tilde{\boldsymbol{r}} '} _x L^{\tilde{\boldsymbol{p}} '} _{\xi} ) } \\
& \lesssim \|\phi\|_{L^{\frac{d}{d-1}} _{x, \xi}} \|Q^+ (f, f)\|_{L^{\tilde{\boldsymbol{q}} '} ([\tau, t]; L^{\tilde{\boldsymbol{r}} '} _x L^{\tilde{\boldsymbol{p}} '} _{\xi} ) } 
\end{align*}
and duality argument yield
\begin{align*}
\left\| \int_{0}^{t} U(-s) Q^{+} (f, f) (s)ds - \int_{0}^{\tau} U(-s)Q^{+} (f, f) (s)ds \right\|_{L^d _{x, \xi}} \lesssim \|Q^+ (f, f)\|_{L^{\tilde{\boldsymbol{q}} '} ([\tau, t]; L^{\tilde{\boldsymbol{r}} '} _x L^{\tilde{\boldsymbol{p}} '} _{\xi} ) } \to 0
\end{align*}
as $\tau \to t$. This implies $t \mapsto \int_{0}^{t} U(-s) Q^{+} (f, f) (s)ds \in L^d _{x, \xi}$ is continuous and also $t \mapsto \int_{0}^{t} U(t-s) Q^{+} (f, f) (s)ds \in L^d _{x, \xi}$ is. Hence $f \in C ([0, \infty); L^d _{x, \xi})$. The above estimate also ensures the existence of
\begin{align*}
\int_{0}^{\infty} U(-s) Q^{+} (f, f) (s)ds = \lim_{t \to \infty} \int_{0}^{t} U(-s) Q^{+} (f, f) (s)ds
\end{align*} 
in $L^d _{x, \xi}$. Therefore
\begin{align*}
U(-t)f(t) = f_0 + \int_{0}^{t} U(-s) Q^{+} (f, f) (s)ds \to f_0 + \int_{0}^{\infty} U(-s) Q^{+} (f, f) (s)ds
\end{align*}
as $t \to \infty$ in $L^d _{x, \xi}$.
\end{proof}

\begin{remark}\label{25121715}
In Proposition \ref{25112308}, no assumption on the sign of $f_0$ and $f$ is made. If we further assume that $f_0 \ge 0$ holds, then the solution also satisfies $f \ge 0$ since $\Phi$ becomes a contraction in $X_{\epsilon, +} = \{ f \in L^{\boldsymbol{q}} _t L^{\boldsymbol{r}} _x L^{\boldsymbol{p}} _{\xi} \mid f \ge 0,  \|f\|_{L^{\boldsymbol{q}} _t L^{\boldsymbol{r}} _x L^{\boldsymbol{p}} _{\xi} } < \epsilon \}$.
\end{remark}
In order to deal with the full equation we need an additional assumption on $f_0$.
\begin{proposition}\label{25121810}
Assume $d=3$ and all the conditions in Proposition \ref{25112308}. We set $a_1 = \frac{15}{8}$. If we further assume $\|f_0\|_{L^{a_1} _{x, \xi}}$ is sufficiently small, the unique solution $f_{+}$ to (\ref{251123140}) constructed in Proposition \ref{25112308} satisfies 
\begin{align*}
\|f_+\|_{L^2 ([0, \infty); L^{\frac{30}{11}} _{x} L^{\frac{10}{7}} _{\xi})} \lesssim \|f_0\|_{L^{a_1} _{x, \xi}},
\end{align*}
where $(q_1, r_1, p_1, a_1) = (2, \frac{30}{11}, \frac{10}{7}, a_1)$ is a KT-admissible pair.
\end{proposition}
\begin{proof}
For $(\boldsymbol{q}, \boldsymbol{r}, \boldsymbol{p}) \in \Lambda$, we set $\frac{1}{\tilde{p_1}'} = \frac{1}{\boldsymbol{p}} + \frac{1}{30}$, $\frac{1}{\tilde{r_1}'} = \frac{1}{\boldsymbol{r}} + \frac{11}{30}$ and $\frac{1}{\tilde{q_1}'} = \frac{1}{\boldsymbol{q}} + \frac{1}{2}$. By Corollary 2.10 in \cite{HJ2}, we have
\begin{align*}
\|Q^{+} (f, g)\|_{L^{\tilde{q_1}'} _t L^{\tilde{r_1}'} _x L^{\tilde{p_1}'} _{\xi}} \lesssim \|f\|_{L^{\boldsymbol{q}} _t L^{\boldsymbol{r}} _x L^{\boldsymbol{p}} _{\xi}} \|g\|_{L^{q_1} _t L^{r_1} _x L^{p_1} _{\xi}}.
\end{align*}
Then we can prove that $\Phi$ in Proposition \ref{25112308} is a contraction on $Y_{\epsilon} = \{ f \in Y := L^{\boldsymbol{q}} _t L^{\boldsymbol{r}} _x L^{\boldsymbol{p}} _{\xi} \cap L^{q_1} _t L^{r_1} _x L^{p_1} _{\xi} \mid \|f\|_{Y} < \epsilon \}$ if $\epsilon >0$ and $\|f_0\|_{L^3 \cap L^{a_1} _{x, \xi}}$ are sufficiently small. Indeed, by the Strichartz estimates (\ref{2412232304}) and (\ref{2412232355}),
\begin{align*}
\|\Phi [f]\|_{L^{q_1} _t L^{r_1} _x L^{p_1} _{\xi}} &\lesssim \|f_0\|_{L^{a_1} _{x, \xi}} + \|Q^{+} (f, f)\|_{L^{\tilde{q_1}'} _t L^{\tilde{r_1}'} _x L^{\tilde{p_1}'} _{\xi}} \\
& \lesssim \|f_0\|_{L^{a_1} _{x, \xi}} + \|f\|_{L^{\boldsymbol{q}} _t L^{\boldsymbol{r}} _x L^{\boldsymbol{p}} _{\xi}} \|f\|_{L^{q_1} _t L^{r_1} _x L^{p_1} _{\xi}}
\end{align*}
holds. Combining with Proposition \ref{25112308}, we obtain
\begin{align*}
\|\Phi [f]\|_{Y} \lesssim \|f_0\|_{L^3 \cap L^{a_1} _{x, \xi}} + \|f\|^2 _{Y} \quad and \quad \|\Phi [f] - \Phi [g]\|_{Y} \lesssim (\|f\|_{Y} + \|g\|_{Y}) \|f-g\|_{Y},
\end{align*}
which indicates our assertion.
\end{proof}
\begin{lemma}\label{25141121}
We define a functional $L$ by $Q^{-} (f, g) = f L(g)$ for $f, g \in C^{\infty} _{0}$. Then
\begin{align}
\|L(g)\|_{L^{\boldsymbol{q} _1} _t L^{\boldsymbol{r} _1} _x L^{\boldsymbol{p} _1} _{\xi}} \lesssim \|g\|_{L^{q_1} _t L^{r_1} _x L^{p_1} _{\xi}} \label{25141333}
\end{align}
holds, where $(\boldsymbol{q} _1, \boldsymbol{r} _1, \boldsymbol{p} _1) = (2, \frac{30}{11}, 30)$.
\end{lemma}
\begin{proof}
By Grad's cut-off condition, for $f, g \in C^{\infty} _{0}$,
\begin{align*}
\left| \int fL(g) (t, x, \xi) dtdxd \xi \right| &= \left| \int f(t, x, \xi) \left( \int g(t, x, \xi _{*}) |\xi - \xi _{*}|^{-1} b(\cos \theta) d \omega d \xi _{*} \right) dtdxd\xi \right| \\
& \lesssim \left| \int f(t, x, \xi) \left( \int g(t, x, \xi _{*}) |\xi - \xi _{*}|^{-1} d \xi _{*} \right) dtdxd\xi \right| \\
& \lesssim \int \|f(t, x, \cdot)\|_{{\boldsymbol{p} _1}'} \|g(t, x, \cdot)\|_{p_1} dtdx \\
& \lesssim \|f\|_{L^{{\boldsymbol{q} _1}'} _t L^{{\boldsymbol{r} _1}'} _x L^{{\boldsymbol{p} _1}'} _{\xi}} \|g\|_{L^{\boldsymbol{q} _1} _t L^{\boldsymbol{r} _1} _x L^{p_1} _{\xi}}
\end{align*}
holds, where in the third line we have used the Hardy-Littlewood-Sobolev inequality. Since $q_1 = \boldsymbol{q} _1$ and $r_1 = \boldsymbol{r} _1$ hold, by the duality argument, we have (\ref{25141333}).
\end{proof}
As seen in the proof, $f_0 \ge 0$ is required to construct monotone sequences $\{g_n\}$ and $\{h_n\}$.
\begin{proof}[Proof of the existence of solution in Theorem \ref{25121757}]
(Step 1) Based on the Kaniel-Shinbrot iteration argument (\cite{KS, CDP, HJ2}) we set
\begin{align*}
&h_1 (t) \equiv 0, \quad g_1 (t) = U(t)f_0 + \int_{0}^{t} U(t-s) Q^{+} (f_{+}, f_{+}) (s)ds = f_+, \\
&h_2 (t) = U(t)f_0 e^{-\int_{0}^{t} U(t-s) L(g_1) (s)ds}, \quad g_2 (t) = U(t)f_0 + \int_{0}^{t} U(t-s) Q^{+} (g_1, g_1) (s)ds = f_+.
\end{align*}
Note that $h_2$ is well-defined since Proposition \ref{25121810} and Lemma \ref{25141121} yield $L(g_1) \in L^{\boldsymbol{q} _1} _t L^{\boldsymbol{r} _1} _x L^{\boldsymbol{p} _1} _{\xi}$. The important points are that $0 \le h_1 (t) \le h_2 (t) \le g_2 (t) \le g_1 (t) \le f_{+}$ and 
\[
\left\{
\begin{array}{l}
(\partial _{t} + H_p ) g_2 (t, x, \xi) + L(h_1)g_2 (t, x, \xi) = Q^{+} (g_1, g_1) (t, x, \xi) \\
(\partial _{t} + H_p ) h_2 (t, x, \xi) + L(g_1)h_2 (t, x, \xi) = Q^{+} (h_1, h_1) (t, x, \xi) \\
g_2 (0) = h_2 (0) = f_0 \\
\end{array}
\right.
\]
hold in the distributional sense. The first one is a consequence of Remark \ref{25121715}, $f_0 \ge 0$ and $b \ge 0$. Now we iteratively define
\begin{align*}
&h_{n+1} (t) = U(t)f_0 e^{-\int_{0}^{t} U(t-s) L(g_n)(s) ds} + \int_{0}^{t}  e^{-\int_{s}^{t} U(t-\tau) L(g_n)(\tau) d\tau} U(t-s) Q^{+} (h_n, h_n) (s) ds \\
&g_{n+1} (t) = U(t)f_0 e^{-\int_{0}^{t} U(t-s) L(h_n)(s) ds} + \int_{0}^{t}  e^{-\int_{s}^{t} U(t-\tau) L(h_n)(\tau) d\tau} U(t-s) Q^{+} (g_n, g_n) (s) ds.
\end{align*}
By definition, if we assume $0 \le h_n \le h_{n+1} \le g_{n+1} \le g_{n} \le f_+$, we obtain $0 \le h_n \le h_{n+1} \le h_{n+2} \le g_{n+2} \le g_{n+1} \le g_n \le f_+$. Hence $0 \le h_1 \le h_n \le h_{n+1} \le g_{n+1} \le g_{n} \le f_+$ holds for all $n \in \N$ by an induction argument. Notice that $h_{n+1}$ and $g_{n+1}$ satisty
\[
\left\{
\begin{array}{l}
(\partial _{t} + H_p ) g_{n+1} (t, x, \xi) + L(h_n)g_{n+1} (t, x, \xi) = Q^{+} (g_n, g_n) (t, x, \xi) \\
(\partial _{t} + H_p ) h_{n+1} (t, x, \xi) + L(g_n)h_{n+1} (t, x, \xi) = Q^{+} (h_n, h_n) (t, x, \xi) \\
g_{n+1} (0) = h_{n+1} (0) = f_0 \\
\end{array}
\tag{$B_n$}\label{25112314}
\right.
\]
in the distributional sense. By the monotonicity of $h_n$ and $g_n$ we have $h := \lim_{n \to \infty} h_n$ and $g := \lim_{n \to \infty} g_n$, which satisfy $0 \le h \le g \le f_+$. Hence the convergence is not only pointwise but also in $L^{\boldsymbol{q}} _t L^{\boldsymbol{r}} _x L^{\boldsymbol{p}} _{\xi} \cap L^{q_1} _t L^{r_1} _x L^{p_1} _{\xi}$ and in the distributional sense. We remark that a similar argument as in Proposition \ref{25112308} yields $Q^{+} (h, h), Q^{+} (g, g) \in L^{\tilde{\boldsymbol{q}} '} _t L^{\tilde{\boldsymbol{r}} '} _x L^{\tilde{\boldsymbol{p}} '} _{\xi}$ and Corollary 2.10 in \cite{HJ2} gives $Q^{\pm} (h, h), Q^{\pm} (g, g) \in L^{\tilde{q_1}'} _t L^{\tilde{r_1}'} _x L^{\tilde{p_1}'} _{\xi}$. Now taking $n \to \infty$ in (\ref{25112314}), we obtain
\[
\left\{
\begin{array}{l}
(\partial _{t} + H_p ) g (t, x, \xi) + L(h)g (t, x, \xi) = Q^{+} (g, g) (t, x, \xi) \\
(\partial _{t} + H_p ) h (t, x, \xi) + L(g)h (t, x, \xi) = Q^{+} (h, h) (t, x, \xi) \\
g (0) = h (0) = f_0 \\
\end{array}
\tag{$B_{\infty}$}\label{25161308}
\right.
\]
in the distributional sense. To prove $g \equiv h$, we set $w := g-h \ge 0$. Then we have
\begin{align}
(\partial _{t} + H_p ) w(t, x, \xi) = Q^{+} (g, w) + Q^{+} (w, h) + Q^{-} (g, w) - Q^{-} (w, g) \label{25151600}
\end{align}
and $w(0) = 0$. Suppose $\|w\|_{L^{q_1} _t L^{r_1} _x L^{p_1} _{\xi}} \neq 0$. Then there exists $T \in (0, \infty)$ such that $\|w\|_{L^{q_1} ([0, T]; L^{r_1} _x L^{p_1} _{\xi})}  > 0$. We set $T_0 = \inf \{ t \in [0, T] \mid \|w\|_{L^{q_1} ([0, t]; L^{r_1} _x L^{p_1} _{\xi})}  > 0\}$. By definition $w(t) =0$ if $t \in [0, T_0]$. By integrating (\ref{25151600}) we have
\begin{align*}
0 \le w(t) &= \int_{0}^{t} U(t-s) \left(Q^{+} (g, w) + Q^{+} (w, h) + Q^{-} (g, w) - Q^{-} (w, g) \right)(s) ds \\
& \le \int_{0}^{t} U(t-s) \left(Q^{+} (g, w) + Q^{+} (w, h) + Q^{-} (g, w) \right)(s) ds.
\end{align*}
Hence (\ref{2412232355}) yields
\begin{align*}
\|w\|_{L^{q_1} ([T_0, s]; L^{r_1} _x L^{p_1} _{\xi})} &\lesssim \|Q^{+} (g, w) + Q^{+} (w, h) + Q^{-} (g, w) \|_{L^{\tilde{q_1}'} ([T_0, s]; L^{\tilde{r_1}'} _x L^{\tilde{p_1}'} _{\xi} )} \\
& \lesssim ( \|g\|_{L^{\boldsymbol{q}} ([T_0, s]; L^{\boldsymbol{r}} _x L^{\boldsymbol{p}} _{\xi})} + \|h\|_{L^{\boldsymbol{q}} ([T_0, s]; L^{\boldsymbol{r}} _x L^{\boldsymbol{p}} _{\xi})})\|w\|_{L^{q_1} ([T_0, s]; L^{r_1} _x L^{p_1} _{\xi})}.
\end{align*}
Therefore if $s > T_0$ is sufficiently close to $T_0$, $\|w\|_{L^{q_1} ([T_0, s]; L^{r_1} _x L^{p_1} _{\xi})} \le \frac{1}{2} \|w\|_{L^{q_1} ([T_0, s]; L^{r_1} _x L^{p_1} _{\xi})}$ holds, which indicates $\|w\|_{L^{q_1} ([0, T_0 +s]; L^{r_1} _x L^{p_1} _{\xi})} =0$. This contradicts the definition of $T_0$ and hence $w \equiv 0$, i.e. $g \equiv h$. By (\ref{25161308}), $g$ satisfies $(\partial _{t} + H_p ) g =Q(g, g)$ and integrating this equation gives
\begin{align*}
g(t) = U(t)f_0 + \int_{0}^{t} U(t-s) Q(g, g) (s)ds.
\end{align*}
Now we show $g \in C([0, \infty); L^3 _{x, \xi})$. By Proposition \ref{25112308}, it suffices to show
\begin{align*}
v(t) = \int_{0}^{t} U(-s) Q^{-} (g, g) (s)ds \in C([0, \infty); L^3 _{x, \xi}).
\end{align*}
By a similar argument as in the proof of the continuity in Proposition \ref{25112308}, $v \in C([0, \infty); L^{\frac{15}{8}} _{x, \xi})$. Suppose $v(t)$ is not continuous as a $L^3 _{x, \xi}$-valued function at $t = \tilde{t}$. Then we have $\delta >0$ and $t_n \nearrow \tilde{t}$ or $t_n \searrow \tilde{t}$ such that $\|v(t_n)-v(\tilde{t})\|_{L^3 _{x, \xi}} \ge \delta$. Since $\mp (v(t_n)-v(\tilde{t}))$ is nonnegative and monotonously decreasing, there exists nonnegative $u \in L^3 _{x, \xi}$ satisfying $\mp (v(t_n)-v(\tilde{t})) \to u$ $a.e.$ and in $L^3 _{x, \xi}$ (Note $0 \le v(t) \le \int_{0}^{t} U(-s) Q^{+} (g, g) (s)ds \le \int_{0}^{\infty} U(-s) Q^{+} (g, g) (s)ds \in L^3 _{x, \xi}$). Therefore $\|u\|_{L^3 _{x, \xi}} \ge \delta$ holds. However $v \in C([0, \infty); L^{\frac{15}{8}} _{x, \xi})$ implies $\langle v(t_n) - v(\tilde{t}), \phi \rangle \to 0$ for $\phi \in C^{\infty} _{0}$ and hence $\langle u, \phi \rangle =0$. This contradicts $\|u\|_{L^3 _{x, \xi}} \ge \delta$ and we obtain $v \in C([0, \infty); L^3 _{x, \xi})$.

(Step 2) We show the uniqueness of the solution. Assume we have two solutions $g$ and $h$. Then $w=g-h$ satisfies
\[
\left\{
\begin{array}{l}
(\partial _{t} + H_p ) w + L(g) w = Q^{+} (g, w) + Q^{+} (w, h) - Q^{-} (h, w) \\
w (0) = 0. \\
\end{array}
\right.
\]
By integrating this equation we obtain
\begin{align*}
|w(t)| &= \left| \int_{0}^{t} e^{-\int_{s}^{t} U(t-\tau) L(g) (\tau) d\tau} U(t-s) [Q^{+} (g, w) + Q^{+} (w, h) - Q^{-} (h, w)](s) ds \right| \\
& \lesssim \int_{0}^{t} U(t-s) [Q^{+} (g, |w|) + Q^{+} (|w|, h) + Q^{-} (h, |w|)](s) ds.
\end{align*}
Then a similar argument as in (Step 1) yields $w \equiv 0$.
\end{proof}

\begin{proof}[Proof of the scattering in Theorem \ref{25121757}]
As in Proposition \ref{25112308}, it suffices to show that the outcome of the loss term: $\int_{0}^{t} U(-s) Q^{-} (g, g) (s)ds$ converges in $L^3 _{x, \xi}$ as $t \to \infty$ (The convergence of $\int_{0}^{t} U(-s) Q^{+} (g, g) (s)ds$ is essentially proven in Proposition \ref{25112308}). Since $\int_{0}^{t} U(-s) Q^{-} (g, g) (s)ds \le f_0 + \int_{0}^{\infty} U(-s) Q^{+} (g, g) (s)ds \in L^3 _{x, \xi}$, there exists $\tilde{u} \in L^3 _{x, \xi}$ such that
\begin{align*}
\int_{0}^{n} U(-s) Q^{-} (g, g) (s)ds \to \tilde{u} \quad as \quad n \to \infty
\end{align*}
in $L^3 _{x, \xi}$ and $a.e.$, where the monotonicity is also used. Then if $t >n$,
\begin{align*}
\left\| \tilde{u} - \int_{0}^{t} U(-s) Q^{-} (g, g) (s)ds \right\|_{L^3 _{x, \xi}} \le \left\| \tilde{u} - \int_{0}^{n} U(-s) Q^{-} (g, g) (s)ds \right\|_{L^3 _{x, \xi}}
\end{align*}
holds and this implies $\int_{0}^{t} U(-s) Q^{-} (g, g) (s)ds \to \tilde{u}$ as $t \to \infty$.
\end{proof}
\begin{remark}\label{25171551}
The Strichartz estimates also have applications to chemotaxis systems. Here we comment on a model of chemotaxis by Othmer, Dunbar and Alt (see \cite{BCGP} for details):
\[
\left\{
\begin{array}{l}
\partial _{t} f(t, x, \xi) +H_p f(t, x, \xi) = \int_{\Xi} (T[S]f' - T^* [S]f) d \xi' \\
f(0, x, \xi) = f_0 (x, \xi) \\
S-\Delta S = \rho := \int_{\Xi} f(t, x, \xi) d\xi.
\end{array}
\right.
\tag{Ch}\label{252131409}
\]    
Here $f' = f(t, x, \xi')$, $T^* [S] (t, x, \xi, \xi') =T[S] (t, x, \xi', \xi)$ and $\Xi \subset \R^3$ is a bounded $3$-dimensional region. If $p(x, \xi) = |\xi|^2$, small data global existence is proved under various assumptions on $T[S]$ in \cite{BCGP}. Once local well-posedness of (\ref{252131409}) is proved, we can prove small data global existence for $p(x, \xi) = g^{ij} (x) \xi_{i} \xi_{j}$ under the condition \cite[(1.5) or (1.6)]{BCGP} on $T[S]$, where $g$ is a nontrapping scattering metric. The proof is just using the Strichartz estimates proved in this paper like \cite[Theorem 3]{BCGP} and we omit details.
\end{remark}
\appendix
\section{Estimate on oscillatory integral}\label{251252117}
In this section we give a proof of the following lemma for the sake of completeness since \cite{V} p.204 seems to be insufficient in the present case.
\begin{lemma}\label{251252125}
If $\re \omega = -1$, the following estimate holds uniformly in $\epsilon \in (0, 1)$.
\begin{align*}
\|\mathcal{F} [t^{\omega} \chi_{\{\epsilon < |t|< \frac{1}{\epsilon}\}}]\|_{\infty} \lesssim \langle \im \omega \rangle e^{\pi |\im \omega|}.
\end{align*}
\end{lemma}
\begin{proof}
It suffices to show the bound
\begin{align*}
\left| \int_{\epsilon <|x|< \frac{1}{\epsilon}} e^{ix \xi} \frac{dx}{x^{1-i\gamma}} \right| \lesssim \langle \gamma \rangle e^{\pi |\gamma |}
\end{align*}
uniformly in $\epsilon \in (0, 1)$, $\gamma \in \R$ and $\xi \in \R$. We assume $\xi >0$ for simplicity ($\xi <0$ can be treated similarly).Then $\left| \int_{\epsilon <|x|< \frac{1}{\epsilon}} e^{ix \xi} \frac{dx}{x^{1-i\gamma}} \right| = \left| \int_{\xi \epsilon < |t| < \xi / \epsilon} e^{it} \frac{dt}{t^{1-i\gamma}} \right|$ holds. If $\xi / \epsilon >1$,
\begin{align*}
\left| \int_{1<|t|< \frac{\xi}{\epsilon}} \frac{e^{it}}{t^{1-i\gamma}} dt \right| = \left| \int_{1<|t|< \frac{\xi}{\epsilon}} \left( \frac{d}{dt} e^{it} \right)\frac{dt}{t^{1-i\gamma}} \right| \lesssim e^{\pi |\gamma |} + \langle \gamma \rangle \left| \int_{1<|t|< \frac{\xi}{\epsilon}} \frac{e^{it}}{t^{2-i\gamma}} dt \right| \lesssim \langle \gamma \rangle e^{\pi |\gamma |}
\end{align*} 
holds by integrating by parts. If $\xi \epsilon <1$, we decompose $ \int_{\xi \epsilon<|t|< 1} \frac{e^{it}}{t^{1-i\gamma}} dt = \int_{\xi \epsilon<|t|< 1} \frac{e^{it} -1}{t^{1-i\gamma}} dt + \int_{\xi \epsilon<|t|< 1} \frac{1}{t^{1-i\gamma}} dt$. The integrand in the first term is bounded by $e^{\pi |\gamma|}$. The second term is estimated as $\left| \int_{\xi \epsilon<|t|< 1} \frac{1}{t^{1-i\gamma}} dt \right| = \left| \frac{1-e^{-\pi \gamma}}{\gamma} + \frac{(\xi \epsilon)^{i\gamma} (e^{-\pi \gamma} -1)}{\gamma} \right| \lesssim e^{\pi | \gamma |}$. The other cases are similarly done so we omit details. 
\end{proof}

\section*{Acknowledgement}
The author would like to thank Kenichi Ito, Haruya Mizutani and Kouichi Taira for useful comments and discussions. He is partially supported by FoPM, WINGS Program, the University of Tokyo.

\address{Graduate School of Mathematical Sciences, The University of Tokyo, 3-8-1 Komaba, Meguro-ku, Tokyo 153-8914, Japan}

\email{Email address: hoshiya@ms.u-tokyo.ac.jp}

\end{document}